\documentclass[10pt]{amsart}

\usepackage{amssymb, paralist, xspace, graphicx, url, amscd, euscript, mathrsfs, stmaryrd}
\usepackage[all]{xy}

\numberwithin{equation}{section}
1
\numberwithin{subsection}{section}
\allowdisplaybreaks[1]

\newenvironment{enumeratea} {\begin{enumerate}[\upshape (a)]} {\end{enumerate}}
\newenvironment{enumeratei} {\begin{enumerate}[\upshape (i)]} {\end{enumerate}}
\newenvironment{enumerate1} {\begin{enumerate}[\upshape (1)]} {\end{enumerate}}

\theoremstyle{definition}
\newtheorem{defn}[subsection]{Definition}
\newtheorem{theorem}[subsection]{Theorem}
\newtheorem{lemma}[subsection]{Lemma}
\newtheorem{example}[subsection]{Example}
\newtheorem{thm}[subsection]{Theorem}
\newtheorem{cor}[subsection]{Corollary}
\newtheorem{lem}[subsection]{Lemma}
\newtheorem{prop}[subsection]{Proposition}
\newtheorem*{props}{Main Properties}
\newtheorem*{defns}{Definition}

\theoremstyle{remark}

\newtheorem{remark}[subsection]{Remark}

\newcommand\cA{\mathcal{A}}  \newcommand\cD{\mathcal{D}} \newcommand\cE{\mathcal{E}}\newcommand\cF{\mathcal{F}} \newcommand\cG{\mathcal{G}}\newcommand\cI{\mathcal{I}}\newcommand\cJ{\mathcal{J}}\newcommand\cL{\mathcal{L}}\newcommand\cM{\mathcal{M}}\newcommand\cO{\mathcal{O}}\newcommand\cS{\mathcal{S}}\newcommand\cU{\mathcal{U}}\newcommand\cW{\mathcal{W}}\newcommand\cX{\mathcal{X}}\newcommand\cY{\mathcal{Y}}\newcommand\cZ{\mathcal{Z}}

    \newcommand\fY{\mathfrak{Y}}\newcommand\fZ{\mathfrak{Z}}

\renewcommand\AA{\mathbb{A}}\newcommand\CC{\mathbb{C}}\newcommand\GG{\mathbb{G}}\newcommand\PP{\mathbb{P}}\newcommand\QQ{\mathbb{Q}}\newcommand\RR{\mathbb{R}}
\newcommand\ZZ{\mathbb{Z}}

\newcommand\id{\mathrm{id}}

\newcommand\aut{\operatorname{Aut}}

\newcommand\spec{\operatorname{Spec}}

\newcommand\arr{\ifinner\to\else\longrightarrow\fi}

\newcommand\hookarr{\hookrightarrow}

\newcommand\im{\operatorname{im}}
\newcommand{\qcoh}{\operatorname{QCoh}}

\newcommand{\coh}{\operatorname{Coh}}

\newcommand{\catsch}[1]{(\mathrm{Sch}/#1)}
\renewcommand{\setminus}{\smallsetminus}

\renewcommand{\v}{\operatorname{v}}
\newcommand{\red}{_{\mathrm{red}}}
\newcommand{\ps}{\operatorname{ps}}
\newcommand{\Supp}{\operatorname{Supp}}

\newcommand{\tors}{\operatorname{tors}}

\newcommand{\what}{\widehat}
\newcommand{\charr}{\operatorname{char}}
\newcommand{\reg}{\operatorname{reg}}
\renewcommand{\ss}{\operatorname{ss}}
\newcommand{\s}{\operatorname{s}}
\newcommand{\rrarrows}{\rightrightarrows}
\newcommand{\coker}{\operatorname{coker}}

\newcommand{\Hom}{\operatorname{Hom}}

\newcommand{\Pic}{\operatorname{Pic}}

\newcommand{\sSpec}{\operatorname{\mathcal{S}pec}}
\newcommand{\sProj}{\operatorname{\mathcal{P}roj}}
\newcommand{\Proj}{\operatorname{Proj}}

\newcommand{\Aut}{\operatorname{Aut}}

\newcommand{\Et}{\operatorname{Et}}
\newcommand{\oh}{\cO}
\newcommand{\Ab}{\operatorname{Ab}}
\newcommand{\Spec}{\spec}

\newcommand{\tensor} {\otimes}

\newcommand{\pre}{\operatorname{pre}}
\newcommand{\Ob}{\operatorname{Ob}}
\newcommand{\Sch}{\operatorname{Sch}}

\newcommand{\iso}{\stackrel{\sim}{\arr}}

\newcommand{\dlim}{{\displaystyle \lim_{\longrightarrow}}\,}

\newcommand{\GL}{\operatorname{GL}}
\newcommand{\PGL}{\operatorname{PGL}}
\newcommand{\SL}{\operatorname{SL}}
\newcommand{\UT}{\operatorname{UT}}
\newcommand{\mapsonto} {\twoheadrightarrow}
\renewcommand{\bar}{\overline}
\newcommand{\gr}{\operatorname{gr}}
\newcommand{\Quot}{\operatorname{Quot}}
\newcommand{\wt} {\widetilde}

\newcommand{\bpf}{\noindent {\em Proof.  }}
\newcommand{\epf}{\qed \vspace{+10pt}}

\begin{document}

\title{Good Moduli Spaces for Artin stacks}

\author[Alper]{Jarod Alper}

\address[Alper]{Department of Mathematics\\
Columbia University\\
2990 Broadway\\
New York, NY 10027\\
U.S.A.}
\email{jarod@math.columbia.edu}


\maketitle

\begin{abstract}
We develop the theory of associating moduli spaces with nice geometric properties to arbitrary Artin stacks generalizing Mumford's geometric invariant theory and tame stacks.
\end{abstract}

\tableofcontents

\section{Introduction}

\subsection{Background}

David Mumford developed geometric invariant theory (GIT) (\cite{git}) as a means to construct moduli spaces.  Mumford used GIT to construct the moduli space of curves and rigidified abelian varieties.  Since its introduction, GIT has been used widely in the construction of other moduli spaces.  For instance,  GIT has been used by Seshadri (\cite{seshadri_bundles}), Gieseker (\cite{gieseker_surface_bundles}), Maruyama (\cite{maruyama_sheaves}), and Simpson (\cite{simpson_sheaves}) to construct various moduli spaces of bundles and sheaves over a variety as well as by Caporaso in \cite{caporaso} to construct a compactification of the universal Picard variety over the moduli space of stable curves.  In addition to being a main tool in moduli theory, GIT has had numerous applications throughout algebraic and symplectic geometry.

Mumford's geometric invariant theory attempts to construct moduli spaces (e.g., of curves) by showing that the moduli space is a quotient of a bigger space parameterizing additional information (e.g. a curve together with an embedding into a fixed projective space) by a reductive group.  In \cite{git}, Mumford systematically developed the theory for constructing quotients of schemes by reductive groups.  The property of reductivity is essential in both the construction of the quotient and the geometric properties that the quotient inherits.  

It might be argued though that the GIT approach to constructing moduli spaces is not entirely natural since one must make a choice of the additional information to parameterize.  Furthermore, a moduli problem may not necessarily be expressed as a quotient. 

Algebraic stacks, introduced by Deligne and Mumford in \cite{deligne-mumford} and generalized by Artin in \cite{artin_versal}, are now widely regarded as the right geometric incarnation of a moduli problem.  A useful technique to study stacks has been to associate to it a coarse moduli space, which retains much of the geometry of the moduli problem, and to study this space to infer geometric properties of the moduli problem.  It has long been folklore (\cite{faltings-chai}) that algebraic stacks with finite inertia (in particular, separated Deligne-Mumford stacks) admit coarse moduli spaces.  Keel and Mori gave a precise construction of the coarse moduli space in \cite{keel-mori}.  Recently, Abramovich, Olsson and Vistoli in \cite{tame} have distinguished a subclass of stacks with finite inertia, called \emph{tame stacks}, whose coarse moduli space has additional desired properties such as its formation commutes with arbitrary base change.  Artin stacks without finite inertia rarely admit coarse moduli spaces.  

We develop an intrinsic theory for associating algebraic spaces to arbitrary Artin stacks which encapsulates and generalizes geometric invariant theory.  If one considers moduli problems of objects with infinite stabilizers (e.g. vector bundles), one must allow a point in the associated space to correspond to potentially multiple non-isomorphic objects (e.g. $S$-equivalent vector bundles) violating one of the defining properties of a coarse moduli space.  However, one might still hope for nice geometric and uniqueness properties similar to those enjoyed by GIT quotients.  

\subsection{Good moduli spaces and their properties}

We define the notion of a \emph{good moduli space} (see Definition \ref{defn_good}) which was inspired by and generalizes the existing notions of a good GIT quotient and tame stack (see \cite{tame}).   The definition is strikingly simple:  
\begin{defns} A quasi-compact morphism $\phi: \cX \arr Y$ from an Artin stack to an algebraic space is a \emph{good moduli space} if 
\begin{enumerate1}
\item The push-forward functor on quasi-coherent sheaves is exact.
\item The induced morphism on sheaves $\oh_Y \arr \phi_* \oh_{\cX}$ is an isomorphism.   
\end{enumerate1}
\end{defns}

A good moduli space $\phi: \cX \arr Y$ has a large number of desirable geometric properties.  We summarize the main properties below:
\begin{props}  If $\phi: \cX \arr Y$ is a good moduli space, then:
\begin{enumerate1}
\item $\phi$ is surjective and universally closed (in particular, $Y$ has the quotient topology).
\item Two geometric points $x_1$ and $x_2 \in \cX(k)$ are identified in $Y$ if and only if their closures $\overline{\{x_1\}}$ and $\overline{\{x_2\}}$ in $\cX \times_{\ZZ} k$ intersect.
\item If $Y' \arr Y$ is any morphism of algebraic spaces, then $\phi_{Y'}: \cX \times_Y Y' \arr Y'$ is a good moduli space.
\item If $\cX$ is locally noetherian, then $\phi$ is universal for maps to algebraic spaces.
\item If $\cX$ is finite type over an excellent scheme $S$, then $Y$ is finite type over $S$.
\item If $\cX$ is locally noetherian, a vector bundle $\cF$ on $\cX$ is the pullback of a vector bundle on $Y$ if and only if for every geometric point $x: \Spec k \arr \cX$ with closed image, the $G_x$-representation $\cF \tensor k$ is trivial.
\end{enumerate1}
\end{props}

\subsection{Outline of results in paper}

Good moduli spaces appear to be the correct notion characterizing morphisms from stacks arising from quotients by \emph{linearly reductive groups} to the quotient scheme.   For instance, if $G$ is a linearly reductive group scheme acting linearly on $X \subseteq \PP^n$ over a field $k$, then the morphism from the quotient stack of the semi-stable locucs to the good GIT quotient $[X^{\ss} / G] \arr X^{\ss} // G$ is a good moduli space.

 In section \ref{git_section}, it is shown that this theory encapsulates the geometric invariant theory of quotients by linearly reductive groups.  In fact, most of the results from \cite[Chapters 0-1]{git} carry over to this much more general framework and we argue that the proofs, while similar, are cleaner.  In particular, in section \ref{stability_section} we introduce the notion of stable and semi-stable points with respect to a line bundle which gives an answer to \cite[Question 19.2.3]{lmb}.

With a locally noetherian hypothesis, we prove that good moduli spaces are universal for maps to arbitrary algebraic spaces (see Theorem \ref{universal_thm}) and, in particular, establish that good moduli spaces are unique.  In the classical GIT setting, this implies the essential result that good GIT quotients are unique in the category of algebraic spaces, an enlarged category where quotients by free finite group actions always exist.

Our approach has the advantage that it is no more difficult to work over an arbitrary base scheme.  This offers a different approach to relative geometric invariant theory than provided by Seshadri in \cite{seshadri_reductivity}, which characterizes quotients by \emph{reductive group schemes}.  We note that geometric invariant theory is valid for non-reduced groups schemes as well as non-affine group schemes.

We show that GIT quotients behave well in flat families (see Corollary \ref{git_flat_families}).  We give a quick proof and generalization (see Theorem \ref{matsushima}) of a result often credited to Matsushima stating that a subgroup of a linearly reductive group is linearly reductive if and only if the quotient is affine.  In section \ref{vector_bundle_section}, we give a characterization of vector bundles on an Artin stack that descend to a good moduli space which generalizes a result of Knop, Kraft and Vust.  Furthermore, in section \ref{topology_section}, we give conditions for when a closed point of an Artin stack admitting a good moduli space is in the closure of a point with lower dimensional stabilizer. 

Although formulated differently by Hilbert in 1900, the modern interpretation of Hilbert's 14th problem asks when the algebra of invariants $A^G$ is finitely generated over $k$ for the dual action of a linear algebraic group $G$ on a $k$-algebra $A$.  The question has a negative answer in general (see \cite{nagata_hilbert14}) but when $G$ is linearly reductive over a field, $A^G$ is finitely generated.  We prove the natural generalization to good moduli spaces (see Theorem \ref{good_thm}(\ref{good_finite_type})): if $\cX \arr Y$ is a good moduli space with $\cX$ finite type over an excellent scheme $S$, then $Y$ is finite type over $S$.  We stress that the proof follows directly from a very mild generalization of a result due to Fogarty in \cite{fogarty2} concerning the finite generation of certain subrings.

Finally, we note here the following trivial but yet interesting consequence of the definition of a good moduli space:  if $\pi: \cX \arr S$ is an Artin stack over a noetherian base $S$ admits a good moduli space $\phi: \cX \arr Y$ with $Y$ proper over $S$, then for any coherent sheaf $\cF$ on $\cX$, the higher direct image sheaves $R^i \pi_* \cF$ are finite.

\subsection{Summary}

The main contribution of this paper is the introduction and systematic development of the theory of good moduli spaces.  Many of the fundamental results of Mumford's geometric invariant theory are generalized.  The proofs of the main properties of good moduli spaces are quite natural except for the proof that good moduli spaces are finite type over the base (Theorem \ref{good_thm} (\ref{good_finite_type})) and the proof that good moduli spaces are unique in the category of algebraic spaces  (Theorem \ref{universal_thm}). 

We give a number of examples of moduli stacks in section \ref{good_examples} admitting good moduli spaces including the moduli of semi-stable sheaves and alternative compactifications of $\cM_g$.   In each of these examples, the existence of the good moduli space was already known due to a GIT stability computation, which is often quite involved.  It would be ideal to have a more direct and intrinsic approach to construct the moduli spaces much in the flavor of Keel and Mori's construction of a coarse moduli space.  For instance, in constructing moduli interpretations of log canonical models of $\bar{M}_g$, the GIT stability computation seems beyond our current means.

One could hope that there is a topological criterion for an Artin stack (eg. a weak valuative criterion) together with an algebraic condition (eg. closed points should have a linearly reductive stabilizers) which would guarantee existence of a good moduli space.  Alternatively, one could ask whether the Hilbert-Mumford numerical criterion \cite[Theorem 2.1]{git} can be generalized to this setting to give an intrinsic and practical criteria for the existence of good moduli spaces. 

 It is also interesting to develop a characteristic $p$ generalization of the theory of good moduli spaces characterizing quotients by \emph{reductive} group schemes.  The author is currently considering these questions.

\subsection*{Acknowledgments}   This paper consists of part of my Ph.D. thesis.  I am indebted to my advisor Ravi Vakil for not only teaching me algebraic geometry but for his encouragement to pursue this project.  I would also like to thank Max Lieblich and Martin Olsson for many inspiring conversations and helpful suggestions.  This work has benefited greatly from conversations with Johan de Jong, Andrew Kresch, David Rydh, Jason Starr and Angelo Vistoli.  

\section{Notation}
Throughout this paper, all schemes are assumed quasi-separated.  Let $S$ be a scheme.  Recall that an \emph{algebraic space} over $S$ is a sheaf of sets $X$ on $\catsch{S}_{\Et}$ such that
\begin{enumeratei}
\item $\Delta_{X/S}: X \arr X \times_S X$ is representable by schemes and quasi-compact.
\item There exists an \'etale, surjective map $U \arr X$ where $U$ is a scheme.
\end{enumeratei}
An \emph{Artin stack} over $S$ is a stack $\cX$ over $\catsch{S}_{\Et}$ such that
\begin{enumeratei}
\item $\Delta_{\cX/S}: \cX \arr \cX \times_S \cX$ is representable, separated and quasi-compact.
\item There exists a smooth, surjective map $X \arr \cX$ where $X$ is an algebraic space.
\end{enumeratei}

All schemes, algebraic spaces, Artin stacks and their morphisms will be over a fixed base scheme $S$.  $\qcoh(\cX)$ will denote the category of quasi-coherent $\oh_{\cX}$-modules for an Artin stack $\cX$ while $\coh(\cX)$ will denote the category of coherent $\oh_{\cX}$-modules for a locally noetherian Artin stack $\cX$.

A morphism $f: X \arr Y$ of schemes is \emph{fppf} if $f$ is locally of finite presentation and faithfully flat.  A morphism $f$ is \emph{fpqc} (see  \cite[Section 2.3.2]{vistoli_fga})) if $f$ is faithfully flat and every quasi-compact open subset of $Y$ is the image of a quasi-compact open subset of $X$.  This notion includes both fppf morphisms as well as faithfully flat and quasi-compact morphisms.

We will say $G \arr S$ is an \emph{fppf group scheme} (resp. an \emph{fppf group algebraic space}) if $G \arr S$ is a faithfully flat, finitely presented and separated group scheme (resp. group algebraic space).  If $G \arr S$ is an fppf group algebraic space, then $BG = [S/G]$ is an Artin stack.  The quasi-compactness and separatedness of $G \arr S$ guarantee that the diagonal of $BG \arr S$ has the same property. 

\subsection{Stabilizers and orbits}
Given an Artin stack $\cX$ a morphism $f: T \arr \cX$ from a scheme $T$, we define the \emph{stabilizer} of $f$, denoted by $G_f$ or $\aut_{\cX(T)}(f)$, as the fiber product
 $$ \xymatrix{
 G_f \ar[r] \ar[d]	& T \ar[d]^{f,f}\\
 \cX \ar[r]^{\Delta_{\cX/S}}		& \cX \times_S \cX				. 		
}$$

\begin{prop} \label{BG_map}
There is a natural monomorphism of stacks $BG_f \arr \cX \times_S T$.   If $G_f \arr T$ is an fppf group algebraic space, then this is a morphism of Artin stacks.  \end{prop}

\bpf Since the stabilizer of $(f, id): T \arr \cX \times_S T$ is $G_f$, we may assume $f: S \arr \cX$.  Let $BG_f^{\pre} \arr (\Sch/S)$ be the prestack defined as the category with objects $(Y \arr S)$ and morphisms $(Y \arr S) \arr (Y' \arr S)$ consisting of the data of morphisms $Y \arr Y'$ and $Y \arr G_f$.  Define a morphism of prestacks 
$$\begin{aligned}
	F: BG_f^{\pre} \arr & \cX 
\end{aligned}$$
by $F(g) = f \circ g \in \cX(Y)$ for $(\stackrel{g}{Y \arr S}) \in \Ob BG_f^{\pre}(Y)$.  It suffices to define the image of morphisms over the identity.  If $\alpha \in \Aut_{BG_f^{\pre}(Y)}(Y \stackrel{g}{\arr} S)$ corresponds to a morphism $\tilde \alpha: Y \arr G_f$, then since $ \Aut_{\cX(Y)} (f \circ g) \cong G_f \times_S Y$, we can define $F(\alpha) = (\tilde \alpha, \id) \in G_f \times_S Y (Y)$.  Since $BG_f$ is the stackification of $BG_f^{\pre}$, $F$ induces a natural map $I: BG_f \arr \cX$.   Since $F$ is a monomorphism, so is $I$. \epf

If $f: T \arr \cX$ is a morphism with $T$ a scheme and $X \arr \cX$ is an fppf presentation, we define the orbit of $f$ in $X$, denoted $o_X(f)$, set-theoretically as the image of $X \times_{\cX} T \arr X \times_S T$.  If $G_f \arr T$ is an fppf group scheme, then the orbit inherits the scheme structure given by the cartesian diagram
$$\xymatrix{
o_X(f) \ar[r] \ar[d]	&	X \times_S T \ar[d] \\
BG_f	 \ar[r]		&	\cX \times_S T
}$$

\subsection{Points and residual gerbes}

There is a topological space associated to an Artin stack $\cX$ denoted by $|\cX|$ which is the set of equivalence classes of field valued points endowed with the Zariski topology (see \cite[Ch. 5]{lmb}).  Given a point $\xi \in |\cX|$, there is a canonical substack $\cG_{\xi}$ called the \emph{residual gerbe} and a monomorphism $\cG_{\xi} \arr \cX$.  Let $\bar \xi$ be sheaf attached to $\cG_{\xi}$ (ie. the sheafification of the presheaf of isomorphism classes $T \mapsto [\cG_{\xi}(T)]$) so that $\cG_{\xi} \arr \bar \xi$ is an fppf gerbe.

\begin{prop} (\cite[Thm. 11.3]{lmb}) \label{gerbe_map} 
If $\cX$ is locally noetherian Artin stack over $S$, then any point $\xi \in |\cX|$ is \emph{algebraic}.  That is,
\begin{enumeratei}
\item $\bar \xi \cong \spec k(\xi)$, for some field $k(\xi)$ called the \emph{residue field} of $\xi$.
\item $\cG_{\xi} \arr \cX$ is representable and, in particular, $\cG_{\xi}$ is an Artin stack.
\item $\cG_{\xi} \arr \spec k(\xi)$ is finite type. \epf
\end{enumeratei} 
\end{prop}

If $\cX$ is locally noetherian, $\xi \in |\cX|$ is locally closed (ie. it is closed in $|\cU|$ for some open substack $\cU \subseteq \cX$) if and only if $\cG_{\xi} \arr \cX$ is a locally closed immersion, and $\xi \in |\cX|$ is closed if and only if $\cG_{\xi} \arr \cX$ is a closed immersion.

If $\xi \in |\cX|$ is algebraic, then for any representative $x: \spec k \arr \cX$ of $\xi$, there is a factorization
\begin{equation} \label{BGgerbe}
\xymatrix{
\spec k \ar[r]	& BG_x \ar[r] \ar[d]	& \cG_{\xi} \ar[r] \ar[d]	&\cX\\
			& \spec k \ar[r]		& \spec k(\xi)
}
\end{equation}
where the square is cartesian.  Furthermore, there exists a representative $x: \spec k \arr \cX$ with $k(\xi) \hookarr k$ a finite extension.

Given an fppf presentation $X \arr \cX$, we define the \emph{orbit} of $\xi \in |\cX|$ in $X$, denoted by $O_X(\xi)$, as the fiber product
$$\xymatrix{
O_X(\xi) \ar[r] \ar[d]	& X \ar[d] \\
\cG_{\xi} \ar[r]		& \cX
}$$
Given a representative $x: \spec k \arr \cX$ of $\xi$, set-theoretically $O_X(\xi)$ is the image of $\spec k  \times_{\cX} X \arr X$. 
Let $R = X \times_{\cX} X \stackrel{ s,t}{\rrarrows} X$ be the groupoid representation. If $\tilde  x \in |X|$ is a lift of $x$, then $O_X(\xi) = s (t^{-1}(\tilde x))$ set-theoretically.

If $x: \Spec k \arr \cX$ is a geometric point, let $\xi: \Spec k \arr \cX \times_S k$.  Then $\cG_{\xi} = BG_x$, $k(\xi) = k$, and $o_X(x) = O_{X \times_S k}(x)$, which is the fiber product 
$$\xymatrix{
o_X(x) \ar[r] \ar[d]		& X \times_S k \ar[d]\\
BG_x \ar[r]			& \cX \times_S k
}$$

\begin{defn}  A geometric point $x: \Spec k \arr \cX$ has a \emph{closed orbit} if $BG_x \arr \cX \times_S k$ is a closed immersion. We will say that an Artin stack $\cX \arr S$ has \emph{closed orbits} if every geometric point has a closed orbit. \end{defn}

\begin{remark} If $p:X \arr \cX$ is an fppf presentation and $\cX$ is locally noetherian, then $x: \Spec k \arr \cX$ has closed orbit if and only if $o_X(x) \subseteq X \times_S k$ is closed and $\cX$ has closed orbits if and only if for every geometric point $x: \Spec k \arr X$, the orbit $o_X(p \circ x) \subseteq X \times_S k$ is closed.
\end{remark}

\section{Cohomologically affine morphisms}

In this section, we introduce a notion characterizing affineness for \emph{non-representable} morphisms of Artin stacks in terms of Serre's cohomological criterion.  Cohomologically affineness will be an essential property of the morphisms that we would like to study from Artin stacks to their good moduli spaces. 

\begin{defn}  A morphism $f: \cX \arr \cY$ of Artin stacks is \emph{cohomologically affine} if $f$ is quasi-compact and the functor
$$f_*: \qcoh(\cX) \arr \qcoh(\cY) $$
is exact.
\end{defn}

\begin{remark} Recall that we are assuming all morphisms to be quasi-separated.  If $f$ is quasi-compact, then by \cite[Lem. 6.5(i)]{Olsson-sheaves} $f_*$ preserves quasi-coherence. \end{remark}

\begin{prop} (Serre's criterion) \label{serre_crit}
A quasi-compact morphism $f: X \arr Y$ of algebraic spaces is affine if and only if it is cohomologically affine.
\end{prop}

\bpf
\cite[II.5.2.1, IV1.7.17-18]{ega} handles the case of schemes.  In \cite[III.2.5]{Knutson}, Serre's criterion is proved for separated morphisms of algebraic spaces with $X$ locally noetherian.  It is straightforward to check that the separated hypothesis is not essential in Knutson's argument.  The noetherian hypothesis is removed in \cite{rydh_approx}. \epf

\begin{remark} Clearly, a morphism is cohomologically affine if and only if the higher direct images of quasi-coherent sheaves vanish.  However, this is not equivalent to the vanishing of the higher direct images of quasi-coherent sheaves of ideals.  For instance, let $G$ be a non-trivial semi-direct product $\AA^1 \rtimes \GG_m$ over a field $k$.  Since $G$ is not linearly reductive (see section \ref{sect_lin_red}), $BG \arr \Spec k$ is not cohomologically affine.  However,  one can compute that $H^i(BG, \oh_{BG}) = 0$ for $i > 0$. \end{remark}

The following proposition states that it is enough to check cohomologically affineness on coherent sheaves.

\begin{prop}  \label{coh_aff_coherence}
If $\cX$ is locally noetherian, then a quasi-compact morphism $f: \cX \arr \cY$ is cohomologically affine if and only if the functor $f_*: \coh(\cX) \arr \qcoh(\cY)$ is exact. \end{prop}

\bpf The proof of \cite[Prop. 2.5]{tame} generalizes using \cite[Prop. 15.4]{lmb}.\epf

\begin{defn} An Artin stack $\cX$ is \emph{cohomologically affine} if $\cX \arr \Spec \ZZ$ is cohomologically affine.\end{defn}

\begin{remark}  An Artin stack $\cX$ is cohomologically affine if and only if $\cX$ is quasi-compact and the global sections functor $\Gamma: \qcoh(\cX) \arr \Ab $ is exact.  It is also equivalent to $\cX \arr \Spec \Gamma(\cX, \oh_{\cX})$ being cohomologically affine. \end{remark}

\begin{remark}  By Proposition \ref{serre_crit}, if $\cX$ is a quasi-compact algebraic space, $\cX$ is cohomologically affine if and only if it is an affine scheme.  \end{remark}

\begin{prop} \label{coh_aff_prop} \quad 
\begin{enumeratei}
\item \label{coh_aff_comp} 
Cohomologically affine morphisms are stable under composition. 
\item \label{coh_aff_affine}
Affine morphisms are cohomologically affine.
\item \label{coh_aff_red}
If $f: \cX \arr \cY$ is cohomologically affine, then $f_{\red}: \cX_{\red} \arr \cY_{\red}$ is cohomologically affine.  If $\cX$ is locally noetherian, the converse is true.
\item \label{coh_aff_base}
If $f: \cX \arr \cY$ is cohomologically affine and $S' \arr S$ is any morphism of schemes, then $f_{S'} = \cX_{S'} \arr \cY_{S'}$ is cohomologically affine. 
\end{enumeratei}
Consider a 2-cartesian diagram of Artin stacks:
$$\xymatrix{
\cX' \ar[r]^{f'} \ar[d]^{g'}	& \cY' \ar[d]^g\\
\cX \ar[r]^f				& \cY
}$$
\begin{enumeratei} \setcounter{enumi}{4}
\item \label{coh_aff_descent}
If $g$ is faithfully flat and $f'$ is cohomologically affine, then $f$ is cohomologically affine.
\item \label{coh_aff_qabc} 
If $f$ is cohomologically affine and $g$ is a quasi-affine morphism, then $f'$ is cohomologically affine.
\item \label{coh_aff_bc}
 If $f$ is cohomologically affine and $\cY$ has quasi-affine diagonal over $S$, then $f'$ is cohomologically affine. In particular, if $\cY$ is a Deligne-Mumford stack, then $f$ cohomologically affine implies $f'$ cohomologically affine.\\
\end{enumeratei}
\end{prop}

\noindent \emph{Proof of (\ref{coh_aff_comp})}: If $f: \cX \arr \cY$, $g: \cY \arr \cZ$ are cohomologically affine, then $g \circ f$ is quasi-compact and $(g \circ f)_* = g_*  f_*$ is exact as it is the composition of two exact functors.  
\\
\\
\noindent \emph{Proof of (\ref{coh_aff_descent})}:  Since $g$ is flat, by flat base change the functors $g^* f_*$ and $f'_* g'^*$ are isomorphic.  Since $g'$ is flat, $g'^*$ is exact so the composition $f'_* g'^*$ is exact.  But since $g$ is faithfully flat, we have that $f_*$ is also exact.  Since the property of quasi-compactness satisfies faithfully flat descent, $f$ is cohomologically affine. 
\\
\\
\noindent \emph{Proof of (\ref{coh_aff_affine})}:  Let $f: \cX \arr \cY$ is an affine morphism.  Since the question is Zariski-local on $\cY$, we may assume there exists an fppf cover by an affine scheme $\Spec B \arr \cY$.  By (\ref{coh_aff_descent}), it suffices to show that $\cX \times_{\cY} \Spec B \arr \Spec B$ is cohomologically affine which is clear since the source is an affine scheme.
\\
\\
\noindent \emph{Proof of (\ref{coh_aff_qabc})}:  Suppose first that $g: \cY' \arr \cY$ is a quasi-compact open immersion.  We claim that the adjunction morphism of functors (from $\qcoh(\cY')$ to $\qcoh(\cY')$) $g^* g_* \arr \id$ is an isomorphism.  For any open immersion $i: Y' \hookrightarrow Y$ of schemes and a sheaf $\cF$ of $\oh_{Y'}$-modules, the natural map $i^* i_* \cF \arr \cF$ is an isomorphism.  Indeed, $i^{-1} i_* \cF \cong \cF$ and $i^{-1} \oh_Y  = \oh_{Y'}$ so that $i^* i_* \cF = (i^{-1} i_* \cF) \tensor_{i^{-1} \oh_Y} \oh_{Y'} \cong \cF$.  Let $p: Y \arr \cY$ be a flat presentation with $Y$ a scheme and consider the fiber square
$$\xymatrix{
Y' \ar[r]^i \ar[d]^{p'}		&  Y \ar[d]^p \\
\cY' \ar[r]^g			& \cY
}$$

Let $\cF$ be a quasi-coherent sheaf of $\oh_{\cY'}$-modules.  The morphism $g^* g_* \cF \arr \cF$ is an isomorphism if and only if $p'^* g^* g_* \cF \arr p'^* \cF$ is an isomorphism.  But $p'^* g^* g_* \cF \cong i^* p^* g_* \cF \cong i^* i_* p'^* \cF$ where the last isomorphism follows from flat base change.  The morphisms are canonical so that the composition $ i^* i_* p'^* \cF \arr p'^* \cF$ corresponds to the adjunction morphism which we know is an isomorphism.

Let $0 \arr \cF_1' \arr \cF_2' \arr \cF_3' \arr 0$ be an exact sequence of quasi-coherent $\oh_{\cX'}$-modules.  Let $\cF_3 = g'_* \cF_2 / g'_* \cF_1$ so that $0 \arr g'_* \cF_1' \arr g'_* \cF_2' \arr \cF_3 \arr 0$ is exact.  Note that $g'^*\cF_3 \cong \cF_3'$ since $g'^* g'_* \arr \id$ is an isomorphism.  Since $f$ is cohomologically affine,
$$ 0 \arr f_*g'_* \cF_1' \arr f_* g'_* \cF_2' \arr f_* \cF_3 \arr 0$$
is exact which implies that
$$ 0 \arr g_*f'_* \cF_1' \arr g_* f'_* \cF_2' \arr f_* \cF_3 \arr 0$$
is exact.  Since $g$ is an open immersion and therefore flat,
$$ 0 \arr f'_* \cF_1' \arr  f'_* \cF_2' \arr g^* f_* \cF_3 \arr 0$$
is exact.  But $g^*f_*$ and $f'_* g'^*$ are isomorphic functors so
$$ 0 \arr f'_* \cF_1' \arr  f'_* \cF_2' \arr f'_* \cF_3' \arr 0$$
is exact.  

Suppose now that $g$ is an affine morphism.  We will use the easy fact:

{\it Sublemma:}  If $g: \cY' \arr \cY$ is an affine morphism and $\cF_1, \cF_2, \cF_3$ are quasi-coherent $\oh_{\cY'}$-Modules, then $\cF_1 \arr \cF_2 \arr \cF_3$ is exact if and only if $g_* \cF_1 \arr g_* \cF_2 \arr g_* \cF_3$ is exact.

{\it Proof of sublemma:}  The question is Zariski-local on $\cY$ so we may assume $\cY$ is quasi-compact.  Let $h: \Spec B \arr \cY$ be an fppf presentation.  There is 2-cartesian square
$$\xymatrix{
\Spec A \ar[r]^{g'} \ar[d]^{h'}		& \Spec B \ar[d]^h \\
\cY' \ar[r]^g						& \cY
}$$
and 
$$\begin{aligned}
\cF_1 \arr \cF_2 \arr \cF_3  \textrm{ exact } & \iff  h'^* \cF_1 \arr h'^* \cF_2 \arr h'^* \cF_3  \textrm{ exact }   \\
 & \iff  g'_* h'^* \cF_1 \arr  g'_* h'^* \cF_2 \arr g'_* h'^* \cF_3   \textrm{ exact } \\
 & \iff  h^* g_* \cF_1 \arr h^* g_* \cF_2 \arr h^* g_* \cF_3  \textrm{ exact } \\
  & \iff  g_* \cF_1 \arr g_* \cF_2 \arr g_* \cF_3  \textrm{ exact }
\end{aligned}$$
where we have used the corresponding fact for morphisms of affine schemes, the faithful flatness of $h$ and $h'$, and flat base change. \epf

Since $g$ is affine, both $g$ and $g'$ are cohomologically affine so that the functors $g_*, g'_*,$ and $f_*$ are exact.  Since $f_* g'_* = g_* f'_*$ is exact, by the above sublemma $f'_*$ is exact.  This establishes (\ref{coh_aff_qabc}).
\\
\\
\noindent \emph{Proof of (\ref{coh_aff_base})}:  If $h: S' \arr S$ is any morphism, let $\{S_i\}$ be an affine cover of $S$ and $\{S'_{ij}\}$ an affine cover of $h^{-1}(S_i)$.  Since $f$ is cohomologically affine, by (\ref{coh_aff_qabc}) that $f_{S_i}$ is cohomologically affine and therefore $f_{S'_{ij}}$ is cohomologically affine.  The property of cohomologically affine is Zariski-local so $f_{S'}$ is cohomologically affine.
\\
\\
\noindent \emph{Proof of (\ref{coh_aff_bc})}:  The question is Zariski-local on $S$ so we may assume $S$ is affine.  The question is also Zariski-local on $\cY$ and $\cY'$ so we may assume that they are quasi-compact.  Let $p: Y  \arr \cY$ be a smooth presentation with $Y$ affine.  Since $\Delta_{\cY/S}$ is quasi-affine, $Y \times_{\cY} Y \cong \cY \times_{\cY \times_S \cY} (Y \times_S Y)$ is quasi-affine and $p$ is a quasi-affine morphism.  After base changing by $p: Y \arr \cY$ and choosing a smooth presentation $Z \arr \cY'_Y$ with $Z$ an affine scheme, we have the 2-cartesian diagram:
$${\def\objectstyle{\scriptstyle}
\def\labelstyle{\scriptstyle}
\xymatrix@=20pt{
				& \cZ \ar[d] \ar[rr]^{h''}			&				& Z \ar[d] \\
				&\cX'_Y \ar[rr]^{h'} \ar[dd] \ar[dl]			&				& \cY'_{Y} \ar[dd] \ar[dl] \\
\cX' \ar[rr]^{\qquad f'} \ar[dd]_{g'}&						& \cY' \ar[dd]_{g}	& \\
				& \cX_Y \ar[rr]^{\qquad h} \ar[dl]				& 				& Y \ar[dl]^p \\
\cX  \ar[rr]^{f}	&							& \cY 		&
}}$$
Since $f$ is cohomologically affine and $p$ is a quasi-affine morphism, by (\ref{coh_aff_qabc}) $h$ is cohomologically affine.   The morphism $Z \arr Y$ is affine which implies that $h''$ is cohomologically affine.  Since the composition $Z \arr \cY'_Y \arr \cY'$ is smooth and surjective, by descent $f'$ is cohomologically affine.  

For the last statement, $\Delta_{\cY/S}: \cY \arr \cY \times_S \cY$ is separated, quasi-finite and finite type so by Zariski's Main Theorem for algebraic spaces, 
$\Delta_{\cY/S}$ is quasi-affine. 
\\
\\
\noindent \emph{Proof of (\ref{coh_aff_red})}:  Since $\cX_{\red} \arr \cX$ is affine, the composition $\cX_{\red} \arr \cX \arr \cY$ is cohomologically affine.  Using that $\cY_{\red} \arr \cY$ is a closed immersion, it follows that $\cX_{\red} \arr \cY_{\red}$ is cohomologically affine from the standard property P argument (see Proposition \ref{coh_aff_propP}).  For the converse, it is clear that $f$ is quasi-compact.  We may suppose that $\cX$ is noetherian.  If $\cI$ be the sheaf of ideals of nilpotents in $\oh_{\cX}$, there exists an $N$ such that $\cI^N = 0$.  We will show that for any quasi-coherent sheaf $\cF$, $R^1 f_* \cF = 0$.  By considering the exact sequence,
$$0 \arr \cI^{n+1} \cF \arr \cI^n  \cF \arr \cI^n \cF / \cI^{n+1} \cF \arr 0,$$
and the segment of the long exact sequence of cohomology sheaves
$$R^1f_* \cI^{n+1} \cF \arr R^1f_* \cI^n \cF \arr R^1f_* (\cI^n \cF/ \cI^{n+1} \cF).$$
By induction on $n$, it suffices to show that $R^1 f_* \cI^n \cF/ \cI^{n+1} \cF = 0$.

If $i: \cX_{\red} \hookarr \cX$ and $j: \cY_{\red} \hookarr \cY$, then for each $n$, $\cI^n \cF / \cI^{n+1} \cF = i_* \cG_n$ for a sheaf $\cG_n$ on $\cX_{\red}$ and
$$R^i f_* (\cI^n \cF / \cI^{n+1} \cF) = R^i (f \circ i)_* \cG_n$$
which vanishes if $i > 0$ since 
$f \circ i \simeq j \circ f_{\red}$ is cohomologically affine.  This establishes (\ref{coh_aff_red}).
\epf

\begin{remark} Cohomologically affine morphisms are not stable under arbitrary base change.  For instance, if $A$ is an abelian variety over an algebraically closed field $k$, then $p: \Spec k \arr BA$ is cohomologically affine but base changing by $p$ gives $A \arr \Spec k$ which is not cohomologically affine.  This remark was pointed out to us by David Rydh.
\end{remark}

\begin{remark} It is not true that the property of being cohomologically affine can be checked on fibers.  For instance, $\AA^2 \setminus \{0\} \arr \AA^2$ is not affine but has affine fibers.  While finite morphisms of stacks are necessarily representable morphisms, proper and quasi-finite morphisms need not be.  For a representable morphism, proper and quasi-finite morphisms are finite and thus affine.  However, proper and quasi-finite non-representable morphisms are not necessarily cohomologically affine.  For instance, if $G \arr S$ is a non-linearly reductive finite fppf group scheme (see section \ref{sect_lin_red}), then $BG \arr S$ is proper and quasi-finite but not cohomologically affine.
 \end{remark}

\begin{cor}  \label{coh_aff4} Suppose $\cY$ is an Artin stack with quasi-affine diagonal over $S$.  A morphism $f: \cX \arr \cY$ is cohomologically affine if and only if for all affine schemes $Y$ and morphisms $Y \arr \cY$, the fiber product $\cX \times_{\cY} Y$ is a cohomologically affine stack.
\end{cor}
\bpf If $f$ is cohomologically affine, then for any morphism $Y \arr \cY$, $\cX \times_{\cY} Y \arr Y$ is cohomologically affine.  If $Y$ is affine, $\cX \times_{\cY} Y$ is a cohomologically affine stack.  Conversely, we can assume $\cY$ is quasi-compact so there exists $Y \arr \cY$ a smooth presentation with $Y$ an affine scheme.  Then $\cX \times_{\cY} Y$ being cohomologically affine implies $\cX \times_{\cY} Y \arr Y$ is cohomologically affine which by descent implies $f$ is cohomologically affine. 
\epf

\begin{prop} 
\label{Leray}
If $f: \cX \arr \cY$ is a cohomologically affine morphism of Artin stacks over $S$ and $\cF \in D^+(\cX)$, there is a natural isomorphism $\RR (g \circ f)_* \cF \cong \RR g_*  (f_* \cF)$, where $g: \cY \arr S$ is the structure morphism. 
\end{prop}

\bpf There is a natural isomorphism $\RR (g \circ f)_* \cF \cong \RR g_* \RR f_* \cF$.   Since $f_*$ is exact, $\RR f_* \cF \cong f_* \cF$ in $D^+(\cY)$.  \epf

\begin{prop} \label{coh_aff_propP} 
Let $f: \cX \arr \cY$, $g: \cY \arr \cZ$ be morphisms of Artin stacks over $S$ where either $g$ is quasi-affine or $\cZ$ has quasi-affine diagonal over $S$.  Suppose $g \circ f$ is cohomologically affine and $g$ has affine diagonal.  Then $f$ is cohomologically affine.
\end{prop}

\bpf This is clear from the 2-cartesian diagram
$$\xymatrix{
				& \cX \ar[r]^{(id,f)} \ar[dl]	& \cX \times_{\cZ} \cY \ar[r]^{p_2}	 \ar[dl] \ar[dr] & \cY \ar[dr] \\
\cY \ar[r]^{\Delta}	& \cY \times_{\cZ} \cY	&	&  \cX \ar[r]				& \cZ 
}$$
and Proposition \ref{coh_aff_prop}.
\epf

\subsection{Cohomologically ample and projective}

Let $\cX$ be a quasi-compact Artin stack over $S$ and $\cL$ a line bundle on $\cX$.

\begin{defn} $\cL$ is \emph{cohomologically ample} if there exists a collection of sections $s_i \in \Gamma(\cX, \cL^{N_i})$ for $N_i > 0$ such that the open substacks $\cX_{s_i}$ are cohomologically affine and cover $\cX$.
\end{defn}

\begin{defn} $\cL$ is \emph{relatively cohomologically ample over $S$} if there exists an affine cover $\{S_j\}$ of $S$ such that $\cL|_{\cX_j}$ is cohomologically ample on $\cX_j = \cX \times_S S_j$.
\end{defn}

\begin{remark} Is this equivalent to other notions of ampleness?  The analogue of (a') $\Leftrightarrow$ (c) in \cite[II.4.5.2]{ega} is not true by considering $\oh_{BG}$ on the classifying stack of a linearly reductive group scheme $G$.  The analogue of (a) $\Leftrightarrow$ (a') in \cite[II.4.5.2]{ega} does not hold since for a cohomologically affine stack $\cX$, the open substacks $\cX_f$ for $f \in \Gamma(\cX, \oh_{\cX})$ do not form a base for the topology.    
\end{remark}

\begin{defn} A morphism of $p: \cX \arr S$ is \emph{cohomologically projective} if $p$ is universally closed and finite type, and there exists an $S$-cohomologically ample line bundle $\cL$ on $\cX$.
\end{defn}

\section{Good moduli spaces} \label{good_sect}

We introduce the notion of a good moduli space and then prove its basic properties.  The reader is encouraged to look ahead at some examples in Section \ref{good_examples}.

Let $\phi: \cX \arr Y$ be a morphism where $\cX$ is an Artin stack and $Y$ is an algebraic space.

\begin{defn} \label{defn_good}
 We say that $\phi: \cX \arr Y$ is a \emph{good moduli space} if the following properties are satisfied:
\begin{enumeratei}
\item $\phi$ is cohomologically affine.
\item The natural map $\oh_Y \iso \phi_* \oh_{\cX}$ is an isomorphism. 
\end{enumeratei}
\end{defn}

\begin{remark}  If $\cX$ is an Artin stack over $S$ with finite inertia stack $I_{\cX} \arr \cX$ then by the Keel-Mori Theorem (\cite{keel-mori}) and its generalizations (\cite{conrad}, \cite{rydh_quotients}), there exists a coarse moduli space $\phi: \cX \arr Y$.  Abramovich, Olsson and Vistoli in \cite{tame} define $\cX$ to be a \emph{tame stack} if $\phi$ is cohomologically affine.  Of those Artin stacks with finite inertia, only tame stacks admit good moduli spaces.
\end{remark}

\begin{remark}  A morphism $p: \cX \arr S$ is cohomologically affine if and only if the natural map $\cX \arr \sSpec p_* \oh_{\cX}$ is a good moduli space. \end{remark}

\begin{remark} \label{relative_remark}
 One could also consider the class of arbitrary quasi-compact morphisms of Artin stacks $\phi: \cX \arr \cY$ satisfying the two conditions in Definition \ref{defn_good}.  We call such morphisms \emph{good moduli space morphisms}. Most of the properties below will hold for these more general morphisms.  Precisely, if the target has quasi-affine diagonal, then the analogues of \ref{good_prop1}, \ref{good_base_change}, \ref{ideal_prop1}, \ref{ideal_prop2}, \ref{good_affine_prop}  and \ref{good_thm} (\ref{good_surjective}-\ref{good_separate}, \ref{good_submersive}, \ref{good_geom_conn}-\ref{good_finite_type}) hold.  However, one can only expect uniqueness properties in $\phi$ after requiring $\cY$ to be an algebraic space, or more generally after requiring $\cY$ to be representable over some fixed Artin stack.
\end{remark} 

\begin{prop} \label{good_prop1}  
Suppose $\phi: \cX \arr Y$ is a good moduli space.  Then for any quasi-coherent sheaf $\cF$ of $\oh_Y$-Modules, the adjunction morphism $\cF \arr \phi_* \phi^* \cF$ is an isomorphism.  
\end{prop}

\bpf  If $g: Y' \arr Y$ is a flat morphism, then $\phi': \cX' = \cX \times_Y Y' \arr Y'$ is a good moduli space.  Indeed, Proposition \ref{coh_aff_prop}(\ref{coh_aff_bc}) implies that $\phi'$ is cohomologically affine.  Let $\phi^{\#}: \oh_Y \arr \phi_* \oh_{\cX}$.  By flat base change, $\lambda: g^* \phi_* \oh_{\cX} \arr \phi'_* \oh_{\cX'}$ is an isomorphism.  Since $\phi'^{\#}: \oh_{Y'} \arr \phi'_* \oh_{\cX'}$ is the composition 
$$\oh_{Y'} \cong g^* \oh_Y \stackrel { g^* \phi^{\#}} {\arr} g^* \phi_* \oh_{\cX} \stackrel{\lambda}{\arr} \phi'_* \oh_{\cX'}$$
it follows that $\phi'$ is a good moduli space.  Let $g': \cX' \arr \cX$.  The composition of the pullback via $g$ of the adjunction morphism $\alpha: \cF \arr \phi_* \phi^* \cF$ with the canonical isomorphisms $g^* \phi_* \cong \phi'_* g'^*$ arising from flat base change and $g'^* \phi^* \cong \phi'^* g^*$, 
$$ g^* \cF \stackrel{g^*\alpha}{\arr} g^* \phi_* \phi^* \cF \cong \phi'_* g'^* \phi^* \cF \cong \phi'_* \phi'^* g^* \cF$$
corresponds to the adjunction morphism $g^* \cF \arr \phi'_* \phi'^* g^* \cF$.  Therefore the question is \'etale local in $Y$ so we may assume $Y$ is an affine scheme.

Then any quasi-coherent sheaf $\cF$ on $Y$ has a free resolution $\cG_2 \arr \cG_1 \arr \cF \arr 0$.  Since the adjunction map $\oh_Y \arr \phi_* \phi^* \oh_Y$ is an isomorphism and $\phi^*$ preserves coproducts, $\cG_i \arr \phi_* \phi^* \cG_i$ is an isomorphism.  We have the diagram
$$\xymatrix{
\cG_2 \ar[r] \ar[d]		& \cG_1 \ar[r] \ar[d]		& \cF \ar[r]	\ar[d]			& 0 \\
\phi_* \phi^* \cG_2 \ar[r]	& \phi_* \phi^* \cG_1 \ar[r]		& \phi_* \phi^* \cF \ar[r]	& 0
}$$
where the bottom row is exact because $\phi^*$ is right exact and $\phi_*$ is exact.  Since the left two vertical arrows are isomorphisms, $\cF \arr \phi_* \phi^* \cF$ is an isomorphism. \epf

\begin{remark} The functor $\phi_*$ is not in general faithful.  For example, $\phi: [\AA^1 / \cG_m] \arr \Spec k$ is a good moduli space (see Example \ref{basic_example}) and if $\cI$ is the sheaf of ideals corresponding to the origin, then $\phi_* \cI = 0$.   

For a quasi-coherent sheaf $\cG$ of $\oh_{\cX}$-modules, the adjunction morphism $\phi^* \phi_* \cG \arr \cG$ is not an isomorphism (unless $\phi$ is an isomorphism).  Indeed for any quasi-coherent sheaf $\cF$ on $Y$, $\phi^* \cF$ restricts to trivial representations for all geometric points of $\cX$ (ie. any geometric point $\Spec k \arr \cX$ induces a morphism $i: BG_x \arr \cX$ such that $i^* \phi^* \cF$ corresponds to a trivial representation).  See Section  \ref{vector_bundle_section} for conditions on $\cG$ implying that the adjunction is an isomorphism.
\end{remark}

\begin{prop}  \label{good_base_change}
Suppose
$$\xymatrix{
\cX' \ar[d]^{\phi'} \ar[r]^{g'}		& \cX \ar[d]^{\phi} \\
Y' \ar[r]^{g}				& Y 
}$$
is a cartesian diagram of Artin stacks with $Y$ and $Y'$ algebraic spaces.  Then
\begin{enumeratei}
\item If $\phi: \cX \arr Y$ is a good moduli space, then $\phi': \cX' \arr Y'$ is a good moduli space.
\item If $g$ is fpqc and $\phi': \cX' \arr Y'$ is a good moduli space, then $\phi: \cX \arr Y$ is a good moduli space.
\end{enumeratei}
\end{prop}

\bpf For (ii), Proposition \ref{coh_aff_prop}(\ref{coh_aff_descent}) implies that $\phi$ is cohomologically affine.  The morphism of quasi-coherent $\oh_{\cX}$-modules $\phi^{\#}: \oh_{Y} \arr \phi_* \oh_{\cX}$ pulls back under the fpqc morphism $g$ to an isomorphism so by descent, $\phi^{\#}$ is an isomorphism.

For (i), the property of being a good moduli space is preserved by flat base change as seen in proof of Proposition \ref{good_prop1} and is local in the fppf topology.  Therefore, we may assume $Y= \Spec A$ and $Y' = \Spec A'$ are affine.  There is a canonical identification of $A$-modules $\Gamma(\cX, \phi^* \widetilde{A'}) = \Gamma(\cX \times_A A', \oh_{\cX \times_A A'})$.  By Proposition \ref{good_prop1}, the natural map $A' \arr \Gamma(\cX, \phi^* \widetilde{A'})$ is an isomorphism of $A$-modules.  It follows that $\cX \times_A A' \arr \Spec A'$ is a good moduli space. \epf

\begin{remark} \label{git_remark} Let $S$ be an affine scheme and $\cX = [\Spec A / G]$ with $G$ a linearly reductive group scheme over $S$ (see Section \ref{sect_lin_red}).  Then $\phi: \cX \arr \Spec A^G$ is a good moduli space.  If $g: \Spec B \arr \Spec A^G$.  Then (i) implies that $[\Spec (A \tensor_{A^G} B) / G] \arr \Spec B$ is a good moduli space and in particular $B \cong (A \tensor_{A^G} B)^G$.  If $S = \Spec k$, this is \cite[Fact (1) in Section 1.2]{git}.
\end{remark}

\begin{lem} (Analogue of Nagata's fundamental lemmas) \label{ideal_prop1}
If $\phi: \cX \arr Y$ is a cohomologically affine morphism, then
\begin{enumeratei}
\item For any quasi-coherent sheaf of ideals $\cI$ on $\cX$,
$$\phi_* \oh_{\cX}  / \phi_* \cI \iso \phi_* (\oh_{\cX} / \cI)$$
\item For any pair of quasi-coherent sheaves of ideals $\cI_1, \cI_2$ on $\cX$,
$$\phi_* \cI_1 + \phi_* \cI_2 \iso \phi_*(\cI_1 + \cI_2)$$
\end{enumeratei}
\end{lem}

\bpf Part (i) follows directly from exactness of $\phi$ and the exact sequence $0 \arr \cI \arr \oh_{\cX} \arr \oh_{\cX}/\cI \arr 0$.  For (ii), by applying $\phi_*$ to the exact sequence $0 \arr \cI_1 \arr \cI_1 + \cI_2 \arr \cI_2 / \cI_1 \cap \cI_2 \arr 0$, we have a commutative diagram
$$\xymatrix{
		&				& \phi_* \cI_2 \ar[d] \ar[rd] \\
0 \ar[r]	& \phi_* \cI_1 \ar[r]	& \phi_* (\cI_1 + \cI_2) \ar[r]	& \phi_* \cI_2 / \phi_*(\cI_1+\cI_2) \ar[r] &0
}$$
where the row is exact.  The result follows. \epf

\begin{remark}  \label{infinite_sum}
Part (ii) above implies that for any set of quasi-coherent sheaves of ideals $\cI_{\alpha}$ that
$$\sum_{\alpha} \phi_* \cI_{\alpha} \iso \phi_* \big( \sum_{\alpha} \cI_{\alpha} \big)$$
The statement certainly holds by induction for finite sums and for the general case we may assume that $Y$ is an affine scheme.  For any element $f \in \Gamma(\cX,  \sum_{\alpha} \cI_{\alpha} )$, there exists $\alpha_1, \ldots, \alpha_n$ such that $f \in \Gamma(\cX, \cI_{\alpha_1} + \cdots \cI_{\alpha_n})$ under the natural inclusion so that the statement follows from the finite case.
\end{remark}

\begin{remark} With the notation of Remark \ref{git_remark}, (i) translates into the natural inclusion $A^G/ (I \cap A^G) \hookrightarrow (A/I)^G$ being an isomorphism for any invariant ideal $I \subseteq A$.  Property (ii) translates into the inclusion of ideals  $(I_1 \cap A^G) + (I_2 \cap A^G) \hookrightarrow (I_1 + I_2) \cap A^G$ being an isomorphism for any pair of invariant ideals $I_1,I_2 \subseteq A$.    If $S=\Spec k$, this is precisely \cite[Lemma 5.1.A, 5.2.A]{nagata_invariants-affine} or  \cite[Facts (2) and (3) in Section 1.2]{git}.
\end{remark}

\begin{lem} \label{ideal_prop2}
Suppose $\phi: \cX \arr Y$ is a good moduli space and $\cJ$ is a quasi-coherent sheaf of ideals in $\oh_Y$ defining a closed sub-algebraic space $Y' \hookrightarrow Y$.  Let $\cI$ be the quasi-coherent sheaf of ideals in $\oh_{\cX}$ defining the closed substack $\cX' = Y' \times_Y \cX \hookrightarrow \cX$.  Then the natural map
$$\cJ \arr \phi_* \cI$$
is an isomorphism.
\end{lem}

\bpf Since the property of a good moduli space is preserved under arbitrary base change, $\phi': \cX' \arr Y'$ is a good moduli space.  By pulling back the exact sequence defining $\cJ$, we have an exact sequence $\phi^*\cJ \arr \phi^* \oh_Y \arr \phi^*\oh_{Y'}  \arr 0$. Since the sequence $0 \arr \cI \arr \phi^* \oh_Y \arr \phi^*\oh_{Y'}  \arr 0$ is exact, there is a natural map $\alpha: \phi^* \cJ \arr \cI$.  By composing the adjunction morphism $\cJ \arr \phi_* \phi^* \cJ$ with $\phi_* \alpha$, we have a natural map $\cJ \arr \phi_* \cI$ such that the diagram
$$\xymatrix{
0 \ar[r]	& \cJ \ar[r] \ar[d]		& \oh_Y \ar[r] \ar[d]		& \oh_{Y'} \ar[r] \ar[d]		& 0 \\
0 \ar[r]	& \phi_* \cI \ar[r]	& \phi_* \oh_{\cX} \ar[r]	& \phi_* \oh_{\cX'} \ar[r]	& 0
}$$
commutes and the bottom row is exact (since $\phi_*$ is exact).  Since the two right vertical arrows are isomorphism, $\cJ \arr \phi_* \cI$ is an isomorphism.
\epf

\begin{remark}  With the notation of \ref{git_remark}, this states that for all ideals $I \subseteq A^G$, then $IA \cap A^G = I$.  This fact is used in \cite{git} to prove that if $A$ is noetherian then $A^G$ is noetherian.  We will use this lemma to prove the analogous result for good moduli spaces.
\end{remark}

\begin{lem} \label{good_affine_prop}
Suppose $\phi: \cX \arr Y$ is a good moduli space and $\cA$ is a quasi-coherent sheaf of $\oh_{\cX}$-algebras.  Then $\sSpec_{\cX} \cA  \arr \sSpec_Y \phi_* \cA$ is a good moduli space.  In particular, if $\cZ \subseteq \cX$ is a closed substack and $\im \cZ$ denotes its scheme-theoretic image the morphism $\cZ \arr \im \cZ$ is a good moduli space.
\end{lem}

\bpf By considering the commutative diagram
$$\xymatrix{
\sSpec \cA \ar[r]^i \ar[d]^{\phi'}	& \cX \ar[d]^{\phi}\\
\sSpec \phi_* \cA \ar[r]^j		& Y
}$$
the property P argument of \ref{coh_aff_propP} implies that $\phi'$ is cohomologically affine.  Since $\phi_* i_* \oh_{\Spec \cA} \cong \phi_* \cA$, it follows that  $\oh_{\sSpec \phi_* \cA} \arr \phi'_* \oh_{\sSpec \cA}$ is an isomorphism so that $\phi'$ is a good moduli space.  Let $\cI$ be a quasi-coherent sheaf of ideals in $\oh_{\cX}$ defining $\cZ$.  Then $\cZ \cong \sSpec \oh_{\cX} / \cI$, $\phi_* (\oh_{\cX}/\cI) \cong \phi_* \oh_{\cX} / \phi_* \cI$ and $\phi_* \cI$ is the kernel of $\oh_Y \arr \phi_* i_* \oh_{\cZ}$.
\epf



\begin{lem} \label{product_lem}
If $\phi_1: \cX_1 \arr Y_1$ and $\phi_2: \cX_2 \arr Y_2$ are good moduli spaces, then $\phi_1 \times \phi_2: \cX_1 \times_S \cX_2 \arr Y_1 \times_S Y_2$ is a good moduli space.
\end{lem}

\begin{proof}
The cartesian squares
$$\xymatrix{
				& \cX_1 \times_S \cX_2 \ar[r]^{(\id, \phi_2)} \ar[dl]	& \cX_1 \times_{S} Y_2 \ar[r]^{(\phi_1, \id)}	 \ar[dl] \ar[dr] & Y_1 \times_S Y_2 \ar[dr] \\
\cX_2 \ar[r]	& Y_2 	&	&  \cX_1 \ar[r]				& Y_1
}$$
imply that $(\id, \phi_2)$ and $(\phi_1, \id)$ are good moduli space morphisms (ie. cohomologically affine morphisms $f: \cX \arr \cY$ which induce isomorphisms $\oh_{\cY} \arr f_* \oh_{\cX}$; see Remark \ref{relative_remark}) so the composition $\phi_1 \times \phi_2$ is a good moduli space.
\end{proof}

\begin{thm} \label{good_thm}
If $\phi: \cX \arr Y$ is a good moduli space, then 
\begin{enumeratei}
\item \label{good_surjective} 
	$\phi$ is surjective.
\item \label{good_univ_closed}
	$\phi$ is universally closed.
\item \label{good_separate}
 	If $Z_1, Z_2$ are closed substacks of $\cX$, then
$$  \im Z_1 \cap \im Z_2 = \im (Z_1 \cap Z_2) $$
where the intersections and images are scheme-theoretic.
\item \label{good_orbit_closure_equiv}
For an algebraically closed $\oh_S$-field $k$, there is an equivalence relation defined on $[\cX(k)]$ by $x_1 \sim x_2 \in [\cX(k)]$ if $\overline{ \{x_1\}} \cap \overline{ \{x_2\} } \ne \emptyset$ in $\cX \times_S k$ which induces a bijective map $[\cX(k)]/\kern-4pt\sim \, \, \arr Y(k)$.  That is, $k$-valued points of $Y$ are $k$-valued points of $\cX$ up to closure equivalence.
\item \label{good_submersive}
	$\phi$ is universally submersive (that is, $\phi$ is surjective and $Y$, as well as any base change, has the quotient topology).
\item \label{good_univ_schemes}
	$\phi$ is universal for maps to \underline{schemes} (that is, for any morphism to a scheme $\psi: \cX \arr Z$, there exists a unique map $\xi: Y \arr Z$ such that $\xi \circ \phi = \psi$).
\item \label{good_geom_conn}
	$\phi$ has geometrically connected fibers.
\item \label{good_red}
	$\phi_{\red}: \cX_{\red} \arr Y_{\red}$ is a good moduli space.  If $\cX$ is reduced (resp. quasi-compact, connected, irreducible), then $Y$ is also.  If $\cX$ is locally noetherian and normal, then $Y$ is also.
\item \label{good_flat}
	If $\cX \arr S$ is flat (resp. faithfully flat), then $Y \arr S$ is flat (resp. faithfully flat).
\item \label{good_noeth}
	If $\cX$ is locally noetherian, then $Y$ is locally noetherian and $\phi_*$ preserves coherence.
\item \label{good_finite_type}
	If $S$ is an excellent scheme (see \cite[IV.7.8]{ega}) and $\cX$ is finite type over $S$, then $Y$ is finite type over $S$.
\end{enumeratei}
\end{thm}

\bigskip

\noindent \emph{Proof of (\ref{good_surjective})}:  Let $y: \Spec k \arr Y$ be any point of $Y$.  Since the property of being a good moduli space is preserved under arbitrary base change,
$$\xymatrix{
\cX_y	 \ar[r] \ar[d]^{\phi_y}					& \cX \ar[d]^{\phi} \\
\Spec k \ar[r]^y								& \cY
}$$
$\phi_y: \cX_y \arr \Spec k$ is a good moduli space and $k \iso \Gamma(\cX_y, \oh_{\cX_y})$ is an isomorphism.  In particular, the stack $\cX_y$ is non-empty implying $\phi$ is surjective.   

\bigskip

\noindent \emph{Proof of (\ref{good_univ_closed})}:  If $\cZ \subseteq \cX$ is a closed substack, then Lemma \ref{good_affine_prop} implies that $\cZ \arr \im \cZ$ is a good moduli space.  Therefore, part (\ref{good_surjective}) above implies $\phi(|\cZ|) \subseteq |Y|$ is closed.  Proposition \ref{good_base_change}(ii) implies that $\phi$ is universally closed.

\bigskip

\noindent \emph{Proof of (\ref{good_separate})}:  This is a restatement of Lemma \ref{ideal_prop1}(ii).

\bigskip

\noindent \emph{Proof of (\ref{good_orbit_closure_equiv})}:  We may assume $Y$ and $\cX$ are quasi-compact.  The $\oh_S$-field $k$ gives $s: \Spec k \arr S$. The induced morphism $\phi_s: \cX_s \arr Y_s$ is a good moduli space.  For any geometric point $x \in \cX_s(k)$ and any point $y \in \overline{ \{ x \} }  \subseteq \cX_s$ with $y \in \cX_s(k)$ closed, property (\ref{good_separate}) applied to the closed substacks $\overline{ \{x\} }, \{y\} \subseteq \cX_s$ implies that $\phi_s ( \overline{ \{ x \} } ) \cap \{ \phi_s( y ) \} = \{ \phi_s(y) \}$ and therefore $\phi_s (y) \in \phi_s( \overline{ \{ x \} }) = \overline{ \{ \phi_s(x) \} }$.   But $\phi_s(x)$ and $\phi_s(y)$ are $k$-valued points of $Y_s \arr \Spec k$ so it follows that  $\phi_s(x) = \phi_s(y)$.  This implies both that $\sim$ is an equivalence relation and that $[\cX(k)] \arr Y(k)$ factors into $[\cX(k)]/\sim \, \, \arr Y(k)$ which is surjective.  If $x_1\nsim x_2 \in \cX_s(k)$, then $\overline{ \{x_1\} }$ and $\overline{ \{x_2\} }$ are disjoint closed substacks of $\cX_s$.  By part (\ref{good_separate}), $\phi(\overline{ \{x_1\} })$ and $\phi(\overline{ \{x_2\} })$ are disjoint and in particular $\phi(x_1) \ne \phi(x_2)$. 

\bigskip

\noindent \emph{Proof of (\ref{good_submersive})}: If $Z \subseteq |Y|$ is any subset with $\phi^{-1}(Z) \subseteq |\cX|$ closed.  Then since $\phi$ is surjective and closed, $Z = \phi (\phi^{-1}(Z))$ is closed.   This implies that $\phi$ is submersive and since good moduli spaces are stable under base change, $\phi$ is universally submersive.

\bigskip

\noindent \emph{Proof of (\ref{good_univ_schemes})}:  We adapt the argument of \cite[Prop 0.1 and Rmk 0.5]{git}.  Suppose $\psi: \cX \arr Z$ is any morphism where $Z$ is a scheme. Let $\{V_i\}$ be a covering of $Z$ by affine schemes and set $W_i = |\cX| - \psi^{-1}(V_i) \subseteq |\cX|$.  Since $\phi$ is closed, $U_i = Y - \phi (W_i)$ is open and $\phi^{-1}(U_i) \subseteq \psi^{-1}(V_i)$ for all $i$.  Since $\{\phi^{-1}(V_i)\}$ cover $|\cX|$, $\bigcap_i W_i = \emptyset$ so by Remark \ref{infinite_sum}, $\bigcap_i \phi(W_i) = \emptyset$.  Therefore, $\{U_i\}$ cover $Y$ and  $\phi^{-1}(U_i) \subseteq \psi^{-1}(V_i)$.  Note that for any $\chi: Y \arr Z$ such that $\psi = \chi \circ \phi$, then $\chi(U_i) \subseteq V_i$.  By property (ii) of a good moduli space, we have that $\Gamma(U_i, \oh_{Y}) = \Gamma(\phi^{-1}(U_i), \oh_{\cX})$ so there is a unique map $\chi_i: U_i \arr V_i$ such that   
$$\xymatrix{
\phi^{-1}(U_i) \ar[d]^{\phi} \ar[rd]^{\psi}	\\
U_i \ar@{-->}[r]^{\chi_i}				&	V_i	
}$$
commutes.    By uniqueness $\chi_i = \chi_j$ on $U_i \cap U_j$.   This finishes the proof of (\ref{good_univ_schemes}).  

\bigskip

\noindent \emph{Proof of (\ref{good_geom_conn})}:  For a geometric point $\Spec k \arr Y$, the base change $\cX \times_Y k \arr \Spec k$ is a good moduli space and it separates disjoint closed substacks by (\ref{good_separate}).  Therefore, $\cX \times_Y k$ is connected.

\bigskip

\noindent \emph{Proof of (\ref{good_red})}:  The first statement follows from Proposition \ref{coh_aff_prop}(\ref{coh_aff_red}).  It is easy to check that the properties of being reduced, quasi-compact, connected and irreducible each descend to the good moduli space.  For the final statement, part (\ref{good_noeth}) implies that $Y$ is locally noetherian.  The property of being normal is local in the smooth topology so we may assume $Y$ is a scheme.  Consider the base change for $y \in Y$
$$\xymatrix{
\cU \ar[r] \ar[d]		& \cX \ar[d] \\
\Spec \oh_{Y,y} \ar[r]			& Y
}$$
Then $\cU$ is normal and has a unique closed point.   If $U \arr \cU$ is a smooth presentation with $U$ an affine scheme, then since $U$ is locally noetherian 
and normal, any connected component is integral so that we may assume $U$ is an integral and normal affine scheme.  Since $R = U \times_{\cU} U$ is normal and noetherian, its connected components $R_i$ are integral.  We have
$$\oh_{Y,y} \hookarr \Gamma(U) \stackrel{p_1,p_2}{\rrarrows} \Gamma(R_1) \times \cdots \times \Gamma(R_n)$$
It is clear then that $\oh_{Y,y}$ is an integral domain.  If $c = a/b$ is integral over $\oh_{Y,y}$ with $a,b \in \oh_{Y,y}$, then as $\Gamma(U)$ is integrally closed, $c \in \Gamma(U)$.  We have $p_1(bc) - p_2(bc) = p_1(b) (p_1 (c) - p_2(c)) = 0$.  As each $R_i \arr U$ is dominant, $p_1(b)|_{\Gamma(R_i)} \ne 0$.  It follows that $p_1(c) = p_2(c)$ so $c \in \Gamma(\oh_{Y,y})$. 
\bigskip

\noindent \emph{Proof of (\ref{good_flat})}:  Consider
$$\xymatrix{
\cX \ar[d]^p \ar[r]^{\phi}	&	Y \ar[ld]^q\\
S
}$$
 By Proposition \ref{good_prop1}, the natural map $Id \arr \phi_* \phi^*$ is an isomorphism of functors $\qcoh(Y) \arr \qcoh(Y)$.  Therefore, the composition
$$q^* \iso \phi_* \phi^* q^* \cong \phi_* p^*$$
is an isomorphism of functors $\qcoh(S) \arr \qcoh(\cX)$.  Since $\phi_*$ and $p^*$ are exact, $q^*$ is exact so $q$ is flat.  Clearly, if $p$ is surjective, then $q$ is surjective. 

\bigskip

\noindent \emph{Proof of (\ref{good_noeth})}: Note that $\cX$ is quasi-compact if and only if $Y$ is quasi-compact.  Therefore we may assume $Y$ is quasi-compact so that $\cX$ is noetherian.  The first part follows formally from Proposition \ref{ideal_prop2}.   If $\cJ_{\bullet}: \cJ_1 \subseteq \cJ_2 \subseteq \cdots$ is chain of quasi-coherent ideals in $\oh_Y$, let $\cI_k$ be the coherent sheaf of ideals in $\oh_{\cX}$ defining the closed substack $Y_k \times_Y \cX$, where $Y_k$ is the closed sub-algebraic space defined by $\cJ_k$.  The chain $\cI_{\bullet}: \cI_1 \subseteq \cI_2 \subseteq \cdots$ terminates and therefore $\cJ_{\bullet}$ terminates since $\phi_* \cI_k = \cJ_k$.  Therefore, $Y$ is noetherian.  

For the second statement, we may assume that $Y$ is affine and $\cX$ is irreducible.  We first handle the case when $\cX$ is reduced.  By noetherian induction, we may assume for every coherent sheaf $\cF$ such that $\Supp \cF \subsetneq \cX$, $\phi_* \cF$ is coherent.  Let $\cF$ be a coherent sheaf with $\Supp \cF = |\cX|$.  If $\cF_{\tors}$ denotes the maximal torsion subsheaf of $\cF$ (see \cite[Section 2.2.6]{lieblich_twisted}), then $\Supp \cF_{\tors} \subsetneq \cX$ and the exact sequence
\begin{equation*} \label{torsion_seq}
0 \arr \cF_{\tors} \arr \cF \arr \cF/\cF_{\tors} \arr 0
\end{equation*}
implies $\phi_* \cF$ is coherent as long as $\phi_* (\cF / \cF_{\tors})$ is coherent.  Since $\cF / \cF_{\tors}$ is pure, we may reduce to the case where $\cF$ is pure.  Furthermore, we may assume $\phi_* \cF \ne 0$.  Let $m \ne 0 \in \Gamma(\cX, \cF)$.  We claim that $m: \oh_{\cX} \arr \cF$ is injective.   If $\ker(m) \ne 0$, then $\Supp (\im m) \subsetneq |\cX|$ is a non-empty, proper closed substack which contradicts the purity of $\cF$.  Therefore, we have an exact sequence
$$0 \arr \oh_{\cX} \stackrel{m}{\arr} \cF \arr \cF / \oh_{\cX} \arr 0$$
so that $\phi_* \cF$ is coherent if and only if $\phi_* (\cF / \oh_{\cX})$ is coherent.
Let $p: U \arr \cX$ be a smooth presentation with $U = \Spec A$ affine.  Let $\eta_i \in U$ be the points corresponding to the minimal primes of $A$.  Since $\Spec k(\eta_i) \arr U$ is flat, the sequence
$$0 \arr k(\eta_i) \arr p^* \cF \tensor k(\eta_i) \arr p^*(\cF / \oh_{\cX}) \tensor k(\eta_i) \arr 0$$
is exact so that $\dim_{k(\eta_i)} p^* (\cF / \oh_{\cX}) \tensor k(\eta_i) = \dim_{k(\eta_i)} p^* \cF \tensor k(\eta_i) - 1$.  By induction on these dimensions, $\phi_* \cF$ is coherent.

Finally, if $\cX$ is not necessarily reduced, let $\cJ$ be the sheaf of ideals in $\oh_{\cX}$ defining $\cX_{\red} \hookarr \cX$.  For some $N$, $\cJ^N = 0$.  Considering the exact sequences
$$0 \arr \cJ^{k+1} \cF \arr \cJ^k \cF \arr \cJ^k \cF / \cJ^{k+1} \cF \arr 0$$
Since $\cJ$ annihilates $\cJ^k \cF / \cJ^{k+1} \cF$, $\phi_* (\cJ^k \cF / \cJ^{k+1} \cF)$ is coherent.  It follows by induction that $\phi_* \cF$ is coherent.

\bigskip

\noindent \emph{Proof of (\ref{good_finite_type})}:
Clearly we may suppose $S = \Spec R$ with $R$ excellent and $Y = \Spec A$.  Since $\phi_{\red}: \cX_{\red} \arr Y_{\red}$ is a good moduli space as well as $\phi_{\red}^{-1}(Y_i) \arr Y_i$ for the irreducible components $Y_i$, using \cite[p. 169]{fogarty} we may suppose that $Y$ is integral.    If $A'$ is the integral closure of $A$ in the fraction field of $A$, then since $R$ is excellent, $\Spec A' \arr \Spec A$ is finite and $A'$ is finitely generated over $R$ if and only if $A$ is finitely generated over $R$.  Since $\cX \times_A A' \arr \Spec A'$ is a good moduli space, we may assume $A$ is normal.  

Fogarty proves in \cite{fogarty2} that if $X \arr Y$ is a surjective $R$-morphism with $X$ irreducible and of finite type over $R$ and $Y$ is normal and noetherian, then $Y$ is finite type over $S$.  His argument easily extends to the case where $X$ is not necessarily irreducible but the irreducible components dominate $Y$.  If $p: X \arr \cX$ is any fppf presentation of $\cX$, then $\phi \circ p$ is surjective (from (\ref{good_surjective})) and the irreducible components of $X$ dominate $Y$.  Since $Y$ is normal and noetherian (from (\ref{good_noeth})), Fogarty's result directly implies that $Y$ is finite type over $S$.
\epf

\section{Descent of \'etale morphisms to good moduli spaces}

One cannot expect that an \'etale morphism between Artin stacks induces an \'etale morphism of the associated good moduli spaces.  However, if the morphism induces an isomorphism of stabilizers at a point, then one might expect that \'etaleness is preserved.  The following theorem is a generalization of \cite[Lemma 1 on p.90]{luna} and \cite[Lemma 6.3]{keel-mori} (see \cite[Theorem 4.2] {conrad} for a more transparent and stack-theoretic statement).  We will apply this theorem to prove uniqueness of good moduli spaces in the next section.

\begin{theorem}  \label{etale_preserving}
Consider a commutative diagram $$\xymatrix{ 
\cX \ar[r]^f \ar[d]^{\phi}		& \cX' \ar[d]^{\phi'} \\
Y \ar[r]^g					& Y'
}$$
with $\cX, \cX'$ locally noetherian Artin stacks and $\phi, \phi'$ good moduli spaces and $f$ representable.  Let $\xi \in |\cX|$.  Suppose

\begin{enumeratea}
\item There is a representative $x: \Spec k \arr \cX$ of $\xi$ with $\Aut_{\cX(k)}(x) \hookrightarrow \Aut_{\cX'(k)}(f(x))$ an isomorphism of group schemes.
\item  $f$ is \'etale at $\xi$.
\item $\xi$ and $f(\xi)$ are closed.
\end{enumeratea}
Then $g$ is formally \'etale at $\phi(\xi)$.  
\end{theorem}

\bpf Since $f$ is \'etale at $\xi$, 
there is a cartesian diagram
$$\xymatrix{
\cG_{\xi} \ar[r] \ar[d]	& \cX_1 \ar[r] \ar[d]	& \cdots \\
\cG_{\xi'} \ar[r]		& \cX'_1 \ar[r]		& \cdots
}$$
 where the vertical arrows are \'etale and $\cX_i, \cX'_i$ the nilpotent thickenings of the closed immersions $\cG_{\xi} \hookarr \cX, \cG_{\xi'} \hookarr \cX'$.  Indeed, $\cG_{\xi'} \times_{\cX'} \cX$ is a reduced closed substack of $\cX$ \'etale over $\cG_{\xi'}$ and there is an induced closed immersion $\cG_{\xi} \hookarr \cG_{\xi'} \times_{\cX'} \cX$ which must correspond to the inclusion of the irreducible component of $\{\xi\} \subseteq |\cG_{\xi'} \times_{\cX'} \cX|$.  

Let  $\cX_i \arr Y_i, \cX'_i \arr Y'_i$ be the induced good moduli spaces and $\fY = \dlim Y_i$, $\fY' = \dlim Y'_i$.  The \'etale morphism $\cG_{\xi} \arr \cG_{\xi}$ induces a morphism on underlying sheaves and a diagram
$$\xymatrix{
\cG_{\xi} \ar[r] \ar[d]	& \cG_{\xi'} \ar[d]\\
\Spec k(\xi) \ar[r]	& \Spec k(\xi')
}$$
We claim that the diagram is cartesian and that $k(\xi') \hookarr k(\xi)$ is a separable field extension.  Let $K$ be an algebraic closure of $k(\xi)$.  The morphism $\cG_{\xi} \arr \cG_{\xi'} \times_{k(\xi')} k(\xi)$ pulls back under the base change $\Spec K \arr \Spec k(\xi)$ to the natural map of group schemes $BG_x \arr BG_{f(x)}$, where $x: \Spec K \arr \cX$ is a representative of $\xi$, which by hypothesis (i) is an isomorphism.  This implies that the diagram is cartesian.  Since $\cG_{\xi'} \arr \Spec k(\xi')$ is fppf, descent implies that $\Spec k(\xi) \arr \Spec k(\xi')$ is \'etale.

If $k(\xi) = k(\xi')$, then each $\cX_i \arr \cX'_i$ is an isomorphism which induces an isomorphism $Y_i \arr Y'_i$.  It is clear then that $\what{\oh}_{Y',\phi' \circ f(\xi)} \iso \what{\oh}_{Y,\phi(\xi)}$.

If $Z'_0 = \Spec k(\xi)$, there is an \'etale morphism $h_0: Z'_0 \arr Y'_0$.  There exists unique schemes $Z'_i$ and \'etale morphisms $h_i: Z'_i \arr Y'_i$ such that $Z'_i = Z'_j \times_{Y'_j} Y'_i$ for $i < j$ and inducing a formally \'etale covering $\fZ' \arr \fY'$ with $\fZ' = \dlim Z'_i$.  By base changing by $\fZ' \arr \fY'$, we obtain a formal scheme $\fZ \arr \fY$ with $\fZ= \dlim Z_i$ where $Z_i = Z_i' \times_{Y_i'} Y_i$ and $Z_0 = \bigsqcup \Spec k(\xi)$ as well as a cartesian diagram
$$\xymatrix{
\cG_{\xi} \times_{k(\xi)} Z_0 \ar[r] \ar[d]^{h_0}	&  \cX_1 \times_{Y_1} Z_1	\ar[d]^{h_1} \ar[r]	&  \cX_2 \times_{Y_2} Z_2	\ar[d]^{h_2} \ar[r]		& \cdots \\
\cG_{\xi'} \times_{k(\xi')} k(\xi) \ar[r]		&  \cX'_1 \times_{Y_1'} Z_1' \ar[r] &  \cX'_1 \times_{Y_2'} Z_2' \ar[r] 	& \cdots
}$$
where the vertical arrows are \'etale.  Since $\cG_{\xi} \cong \cG_{\xi'} \times_{k(\xi')} k(\xi)$, the morphism $h_0$ is a disjoint union of isomorphisms.  Since extensions of \'etale morphisms over nilpotent thickenings are unique, each $h_i$ is a disjoint union of isomorphisms.  Therefore, the induced morphism of good moduli spaces $\fZ \arr \fZ'$ is adic and formally \'etale.  In the cartesian diagram
$$\xymatrix{
\fZ \ar[d] \ar[r]	& \fZ' \ar[d]\\
\fY \ar[r]		& \fY'
}$$
the vertical arrows are adic, formally \'etale coverings.  It follows that $\fY \arr \fY'$ is both adic and formally \'etale.  
  \epf

\section{Uniqueness of good moduli spaces}

We will prove that good moduli spaces are universal for maps to algebraic spaces by reducing to the case of schemes (Theorem \ref{good_thm} (\ref{good_univ_schemes})).  

\begin{defn} \label{saturated_def}  If $\phi: \cX \arr Y$ is a good moduli space, an open substack $\cU \subseteq \cX$ is \emph{saturated for $\phi$} if $\phi^{-1} (\phi(\cU)) = \cU$.  
\end{defn}

\begin{remark} If $\cU$ is saturated for $\phi$, then $\phi(\cU)$ is open and $\phi|_{\cU}: \cU \arr \phi(\cU)$ is a good moduli space. 
\end{remark}

\begin{lemma} \label{universal_lem2}
Suppose $\phi: \cX \arr Y$ is a good moduli space.  If $\psi: \cX \arr Z$ is a morphism where $Z$ is a scheme and $V \subseteq Z$ is an open subscheme, then $\psi^{-1}(V)$ is saturated for $\phi$.
\end{lemma}

\bpf If $U=\psi^{-1}(V)$ is not saturated, there exists a $\xi \in  \phi^{-1} (\phi (|U|)) \setminus |U|$ and $\eta \in |U|$ with $\phi(\eta) = \phi(\xi) = y \in |Y|$.  Since $Z$ is a scheme, there exists a morphism $\chi: Y \arr Z$ with $\psi = \chi \circ \phi$. It follows that $\psi(\xi) = \psi(\eta) \in |V|$ which contradicts $\xi \notin |U|$. \epf

The following gives a generalization of \cite[Lemma p.89]{luna} although in this paper, we will only need the special case where $g$ is an isomorphism.

\begin{prop}  \label{finite_prop}
Suppose $\cX, \cX'$ are locally noetherian Artin stacks and
$$\xymatrix{
\cX \ar[r]^f \ar[d]^{\phi}		& \cX' \ar[d]^{\phi'} \\
Y \ar[r]^{g}					& Y'
}$$
is commutative with $\phi,\phi'$ good moduli spaces.  Suppose
\begin{enumeratea}
\item $f$ is representable, quasi-finite and separated.
\item $g$ is finite
\item $f$ maps closed points to closed points.
\end{enumeratea}
Then $f$ is finite.
\end{prop}

\bpf We may assume $S$ and $Y'$ are affine schemes.  By Zariski's Main Theorem (\cite[Thm. 16.5]{lmb}), there exists a factorization
$$\xymatrix{
\cX \ar[r]^I \ar[rd]^f		& \cZ \ar[d]^{f'}\\
					& \cX'
}$$
where $I$ is a open immersion, $f'$ is a finite morphism and $\oh_{\cZ} \hookrightarrow I_* \oh_{\cX}$ is an inclusion.  Since $\cX'$ is cohomologically affine and $f'$ is finite, $\cZ$ is cohomologically affine and admits a good moduli space $\varphi: \cZ \arr Z$.  We have a commutative diagram of affine schemes
$$\xymatrix{
Y \ar[r]^{i}	\ar[rd]^{g}	&  Z \ar[d]^{g'}	&	\Gamma(\cX, \oh_{\cX})		& \Gamma(\cZ, \oh_{\cZ}) \ar[l]_{i^{\#}} \\
&	 Y'	&						& \Gamma(\cX', \oh_{\cX'}) \ar[u]_{g'^{\#}} \ar[ul]_{g^{\#}}.
}$$
Since $i^{\#}$ is injective and $g$ is finite, $i: Y \arr Z$ is a surjective, finite morphism.

For any closed point $\zeta \in |\cZ|$, there exists a closed point $\xi \in |\cX|$ with $\varphi(\zeta) = (i \circ \phi)(\xi)$ and $f(\xi) \in |\cX'|$ is closed.  Then $f'^{-1} (f(\xi)) \subseteq |\cZ|$ is a closed set consisting of finitely many closed points.  In particular, $I(\xi)$ is closed but since $\varphi$ separates closed points and $\varphi( I(\xi)) = \varphi(\zeta)$, it follows that $I(\xi) = \zeta$.  Therefore, $I(\cX)$ contains all closed points.  This implies that $I$ is an isomorphism so that $f$ is finite.
\epf

The following lemma will be useful in verifying condition (iii) above.
\begin{lem}  \label{closed_to_closed_lem}
Suppose
$$\xymatrix{
\cX \ar[r]^f	\ar[rd]^{\phi}	& \cX' \ar[d]^{\phi'}\\
					& Y	
}$$
is a commutative diagram with $\phi, \phi'$ good moduli spaces and $f$ surjective.  Then $f$ maps closed points to closed points.
\end{lem}

\bpf If $\xi \in |\cX|$ is closed, the image $y \in |Y|$ is closed and after base changing by $\Spec k(y) \arr Y$, we have
$$\xymatrix{
\cX_y \ar[r]^{f_y}	\ar[rd]^{\phi_y}	& \cX'_y \ar[d]^{\phi'_y}\\
					& \Spec k(y)
}$$
with $\phi_y$, $\phi'_y$ good moduli spaces.  Since $\cX_y$ and $\cX'_y$ have unique closed points, $f_y(\xi)$ is closed in $|\cX'_y|$ and therefore $f(\xi)$ is closed in $|\cX'|$. \epf

\begin{theorem}  \label{universal_thm}
Suppose $\cX$ is a locally noetherian Artin stack and $\phi: \cX \arr Y$ is a good moduli space.  Then $\phi$ is universal for maps to algebraic spaces. 
\end{theorem}

\bpf Let $Z$ be an algebraic space.  We need to show that the natural map
$$\Hom(Y,Z) \arr \Hom(\cX,Z)$$
is a bijection of sets.  The injectivity argument is functorial by working \'etale-locally on $Y$. 

Suppose $\psi: \cX \arr Z$.  The question is Zariski-local on $Z$ by the same argument in the proof of 
Theorem \ref{good_thm} (\ref{good_univ_schemes}) so we may assume $Z$ is quasi-compact.  There exists an \'etale, quasi-finite surjection $g: Z_1 \arr Z$ with $Z_1$ a scheme.  By Zariski's main theorem for arbitrary algebraic spaces (\cite[Thm 16.5]{lmb} and \cite[Prop. 5.7.8]{raynaud-gruson}), $g$ factors as an open immersion $Z_1 \hookarr \wt{Z}$ and finite morphism $\wt{Z_1}  \arr Z$.  By taking the fiber product by $\psi: \cX \arr Z$, we have 
$$\xymatrix{
					& \wt{ \cX} \ar[d]^{\wt f} \\
\cX_1 \ar[ur]^j	\ar[r]^f	& \cX
}$$
with $j$ an open immersion and $\wt f$ is finite.  By Lemma \ref{good_affine_prop} since $\wt \cX \cong \sSpec \cA$ for a coherent sheaf of $\oh_{\cX}$-algebras $\cA$, there is a good moduli space $\wt \phi: \wt \cX \arr \wt Y$ with $\wt Y = \sSpec \phi_* \cA$.  The induced map $\wt Y \arr Y$ is finite since $\phi_* \cA$ is coherent (Theorem \ref{good_thm} (\ref{good_noeth})).  If $\wt \psi: \wt \cX \arr \wt \cZ$, then $\wt \psi^{-1} (Z_1)$ is saturated for $\wt \phi$ by Lemma \ref{universal_lem2} and therefore there is a good moduli space $\phi_1: \cX_1 \arr Y_1$ inducing a morphism $g: Y_1 \arr Y$ which factors as the composition of the open immersion $Y_1 \hookarr \wt Y$ and the finite morphism $\wt Y \arr Y$.  In particular, $Y_1 \arr Y$ is finite type.

Write $Z_2 = Z_1 \times_Z Z_1$ so that $\bar{s}, \bar{t}: Z_2 \rrarrows Z_1$ is an \'etale equivalence relation and write $\cX_i = \cX \times_Z Z_i$ and $\psi_i: \cX_i \arr Z_i$.  By the above argument, there is a good moduli space $\phi_2: \cX_2 \arr Y_2$ and induced finite type morphisms $s,t: Y_2 \rrarrows Y_1$.  Since $Z_i$ are schemes, there are induced morphisms $\xi_i: \cX_i \arr Z_i$ such that $\xi_i = \psi_i \circ \xi_i$.  By uniqueness, $\chi_1 \circ s = \bar s \circ \chi_2$ and $\chi_1 \circ t = \bar t \circ \chi_2$.  The picture is 
\begin{equation} \label{univ_diag}
\xymatrix{
\cX_2 \ar@<0.5ex>[r] \ar@<-0.5ex>[r] \ar[d]^{\phi_2} 		& \cX_1 \ar[r]^f \ar[d]^{\phi_1}	& \cX \ar[d]^{\phi} \\
Y_2 \ar@<0.5ex>[r]^s \ar@<-0.5ex>[r]_t 	\ar[d]^{\chi_2}		& Y_1 \ar[r]^g \ar[d]^{\chi_1}		& Y \ar@{-->}[d]\\
Z_2 \ar@<0.5ex>[r]^{\bar s} \ar@<-0.5ex>[r]_{\bar t} 			& Z_1 \ar[r] 				& Z
}
\end{equation}
Our goal is to show that $Y_2 \rrarrows Y_1$ is an \'etale equivalence relation with quotient $Y$.  The morphism  $f: \cX_1 \arr \cX$ is surjective, \'etale and preserves stabilizer automorphism groups for all points (in the sense of Theorem \ref{etale_preserving}(a)).  To show that $g: Y_1 \arr Y$ is \'etale, it suffices to check at closed points.  If $y_1 \in |Y_1|$ is closed, then as $g$ is finite type, the image $g(y_1)$ is closed in some open $V \subseteq Y$ and $g$ is \'etale at $y_1$ if and only if $g|_{g^{-1}(V)}$ is \'etale at $y_1$.  We can find a closed point $\xi \in |\phi^{-1}(V)|$ over $g(y_1)$ and a closed preimage $\xi_1 \in |(\phi' \circ g)^{-1}(V)|$ over $y_1$.  It follows from Theorem \ref{etale_preserving} that $g$ is \'etale at $y_1$.  Similarly, $s,t: Y_2 \rrarrows Y_1$ are \'etale.

Now consider the induced 2-commutative diagram
$$\xymatrix{
\cX_1 \ar[d]^{\varphi} \ar[rd]^f		\\
Y_1 \times_Y \cX \ar[r]^h		& \cX
}$$
Then $\varphi$ is \'etale, quasi-compact and separated and, in particular, quasi-finite.  Note that $\varphi$ is also surjective.  Indeed, to check this, we may assume $Y = \Spec K$ for an algebraically closed field $K$ and since $g$ is etale, we may also assume $Y' = \Spec K$ in which case $\varphi$ is isomorphic to $f$ which we know is surjective.  By Lemma \ref{closed_to_closed_lem}, $\varphi$ sends closed points to closed points.  By Corollary \ref{finite_prop}, $\varphi$ is a finite \'etale morphism and since $\varphi$ has only one preimage over any closed point in $Y' \times_Y \cX$, $\varphi$ is an isomorphism.   Similarly $s,t: Y_2 \rrarrows Y_1$ are \'etale and the top squares in diagram \ref{univ_diag} are cartesian.  Furthermore, by universality of good moduli spaces for morphisms to schemes,  $Y_2 = Y_1 \times_Y Y_1$ so that $Y$ is the quotient of the \'etale equivalence relation $Y_2 \rrarrows Y_1$.  Therefore there exists a map $\chi: Y \arr Z$ and the two maps $\chi \circ \phi$ and $\psi$ agree because they agree after \'etale base change.  \epf


\section{Tame moduli spaces}

The following notion captures the properties of a geometric quotient by a linearly reductive group scheme.

\begin{defn} \label{defn_good_coarse}
We will call $\phi: \cX \arr Y$ a \emph{tame moduli space} if
\begin{enumerate}
\item[(i)] $\phi$ is a good moduli space.
\item[(ii)] For all geometric points $\Spec k \arr S$, the map
$$ [\cX(k)] \arr Y(k) $$
is a bijection of sets.
\end{enumerate}
\end{defn}

\begin{remark} $[\cX(k)]$ denotes the set of isomorphism classes of objects of $\cX(k)$. \end{remark}

\begin{remark}  This property is stable under arbitrary base change and satisfies fppf descent.  If $\cX$ is locally noetherian, then by Theorem \ref{universal_thm}, tame moduli spaces are universal for maps to algebraic spaces and therefore $\phi$ is both a good moduli space and coarse moduli space. The map from a tame Artin stack to its coarse moduli space is a tame moduli space. \end{remark}

\begin{prop} \label{universal_homeo}
 If $\phi: \cX \arr Y$ is a tame moduli space, then $\phi$ is a universal homeomorphism.  In particular, $\phi$ is universally open and induces a bijection between open substacks of $\cX$ and open sub-algebraic spaces of $Y$.
\end{prop}

\bpf If $\cU \subset \cX$ is an open substack, let $\cZ$ be the complement.  Since $\phi$ is closed, $\phi(\cZ)$ is closed sub-algebraic space.  Set-theoretically $\phi(\cZ) \cap \phi(\cU) = \emptyset$ because of property (ii) of a tame moduli space.  Therefore, $\phi(\cU)$ is open.
\epf

\begin{prop} \label{coarse_square}
If $\phi: \cX \arr Y$ is a tame moduli space and $x: \Spec k \arr \cX$ is a geometric point, then the natural map $BG_x \arr \cX \times_Y \Spec k$ is a surjective closed immersion.
\end{prop}

\bpf  The morphism $\Spec k \arr \cX \times_S \Spec k$ is finite type so that $BG_x \arr \cX \times_S \Spec k$ is a locally closed immersion.  By considering the cartesian square
$$\xymatrix{
\cX \times_Y k \ar[r] \ar[d]		& \cX \times_S k \ar[d] \\
\Spec k \ar[r]				& Y \times_S k
}$$
it follows since $\Spec k \arr Y \times_S k$ is separated that the induced morphism $BG_x \arr \cX \times_Y k$ is a locally closed immersion.  But it also surjective since $[\cX (k)] \arr Y (k)$ is bijective. \epf

\begin{remark}  It is not true that $BG_x \arr \cX \times_Y \Spec k$ is an isomorphism.  For instance over $S = \Spec k$, if $\cI$ is the ideal sheaf defining $B\GG_m \arr [\AA^1 / \GG_m]$ and $\cX_n \hookarr [\AA^1 / \GG_m]$ is defined by $\cI^{n+1}$ with $n>0$, then $\cX_n \arr \Spec k$ is a good moduli space but the induced map $B \GG_m \arr \cX_n$ is not an isomorphism.
\end{remark}

\begin{prop} (Analogue of \cite[Proposition 0.6 and Amplification 1.3]{git}) \label{good_is_coarse}
Suppose $\phi: \cX \arr Y$ is a good moduli space.  Then $\phi: \cX \arr Y$ is a tame moduli space if and only if $\cX$ has closed orbits.  If this holds and if $Y$ is locally separated, then $Y$ is separated if and only if the image of $\Delta_{\cX/S}: \cX \arr \cX \times_S \cX$ is closed.  
\end{prop}

\bpf  The only if implication is implied by the previous proposition.  Conversely, suppose $\cX$ has closed orbits and suppose $\phi$ is not a tame moduli space.  Let $x_1, x_2 \in \cX(k)$ be two geometric points mapping to $y \in Y(k)$ and $s \in S(k)$.  Since $\phi_s: \cX_s \arr Y_s$ is a good moduli space and $BG_{x_1}, BG_{x_2} \subseteq \cX_s$ are closed substacks with the property that $\phi_s (BG_{x_1}) = \phi_s  (BG_{x_2}) = \{y\} \subseteq |Y|$, it follows that $x_1$ is isomorphic to $x_2$.

Since $\phi: \cX \arr Y$ is a good moduli space, the image of $\Delta_{\cX/S}$ is precisely the image of $\cX \times_Y \cX \arr \cX \times_S \cX$.  Since
$$\xymatrix{
\cX \times_Y \cX \ar[r] \ar[d]	& \cX \times_S \cX \ar[d]^{\phi \times \phi} \\
Y \ar[r]^{\Delta}					&  Y \times_S Y
}$$
is cartesian and $\phi \times \phi$ is submersive, $\Delta(Y)$ is closed if and only if $(\phi \times \phi)^{-1} (\Delta(Y))$ is closed, which is true if and only if $\im(\Delta_{\cX/S})$ is closed. \epf


\subsection{Gluing good moduli spaces}

It is convenient to know when good moduli spaces can be glued together.  Certainly one cannot always expect to glue good moduli spaces (see Example \ref{good_ex_p1}).  Given a cover of an Artin stack by open substacks admitting a good moduli space, one would like criteria guaranteeing the existence of a global good moduli space.

\begin{prop}  \label{good_gluing}
Suppose $\cX$ is an Artin stack (resp. locally noetherian Artin stack) over $S$ containing open substacks $\{\cU_i\}_{i \in I}$ such that for each $i$,  there exists a good moduli space $\phi_i: \cU_i \arr Y_i$ with $Y_i$ a scheme (resp. algebraic space).  Let $\cU = \bigcup_i \cU_i$.  Then there exists a good moduli space $\phi: \cU \arr Y$ and open sub-algebraic spaces $\tilde Y_i \subseteq Y$ such that $\tilde Y_i \cong Y_i$ and $\phi^{-1}(\tilde Y_i) = \cU_i$ if and only if for each $i,j \in I$, $\cU_i \cap \cU_j$ is saturated for $\phi_i: \cU_i \arr Y_i$  (see Definition \ref{saturated_def}).
\end{prop}

\bpf The only if direction is clear.  For the converse, set $\cU_{ij} = \cU_i \cap \cU_j$ and $Y_{ij} = \phi_i(\cU_{ij}) \subseteq Y_i$.  The hypotheses imply that $\phi_i|_{\cU_{ij}}: \cU_{ij} \arr Y_{ij}$ is a good moduli space.  Since good moduli spaces are unique (Theorem \ref{good_thm}(\ref{good_univ_schemes}) and Theorem \ref{universal_thm}), there are unique isomorphisms $\varphi_{ij}: Y_{ij} \iso Y_{ji}$ such that $\varphi_{ij} \circ \phi_i|_{\cU_{ij}} = \phi_j|_{\cU_{ij}}$ and $\varphi_{ij} = \varphi_{ji}^{-1}$.  Set $\cU_{ijk} = \cU_i \cap \cU_j \cap \cU_k$ so that $Y_{ij} \cap Y_{ik} = \phi_i(\cU_{ijk})$.  Since the intersection of saturated sets remains saturated, $\phi_i|_{\cU_{ijk}}: \cU_{ijk} \arr Y_{ij} \cap Y_{ik}$ is a good moduli space and there is a unique isomorphism $\varphi_{ijk}: Y_{ij} \cap Y_{ik} \iso Y_{ji} \cap Y_{jk}$ such that $\varphi_{ijk} \circ \phi_i|_{\cU_{ijk}} = \phi_j|_{\cU_{ijk}}$.  We have $\varphi_{ij}|_{Y_{ij} \cap Y_{ik}} = \varphi_{ijk}$.  The composition
$$\alpha: Y_{ik} \cap Y_{ij} \stackrel{\varphi_{ikj}}{\arr} Y_{ki} \cap Y_{ji} \stackrel{\varphi_{kji}}{\arr} Y_{jk} \cap Y_{ji}$$
satisfies $\alpha \circ \phi_i|_{\cU_{ijk}} = \phi_j|_{\cU_{ijk}}$ so by uniqueness $\varphi_{ijk} = \varphi_{kji} \circ \varphi_{ikj}$.  Therefore, we may glue the $Y_i$ to form a scheme (resp. algebraic space) $Y$.  The morphisms $\phi_i$ agree on the intersection $\cU_{ij}$ and therefore glue to form a morphism $\phi: \cU \arr Y$ with the desired properties. \epf

There is no issue with gluing tame moduli spaces.

\begin{prop}  \label{almost_coarse_glue}
Suppose $\cX$ is an Artin stack (resp. locally noetherian Artin stack) over $S$ containing open substacks $\{\cU_i\}_{i \in I}$ such that for each $i$,  there exists a tame moduli space $\phi_i: \cU_i \arr Y_i$ with $Y_i$ a scheme (resp. algebraic space).  Let $\cU = \bigcup_i \cU_i$.  Then there exists a tame moduli space $\phi: \cU \arr Y$ and open sub-algebraic spaces $\tilde Y_i \subseteq Y$ such that $\tilde Y_i \cong Y_i$ and $\phi^{-1}(\tilde Y_i) = \cU_i$. 
\end{prop}

\bpf By Proposition \ref{universal_homeo}, each $\phi_i$ induces a bijection between open sets of $\cX_i$ and $Y_i$ and therefore every open substack of $\cX_i$ is saturated. \epf

\section{Examples} \label{good_examples}
\begin{example} \label{ex_coarse_space}
If $\cX$ is a tame Artin stack (see \cite{tame}) and $\phi: \cX \arr Y$ is its coarse moduli space, then $\phi$ is a good moduli space.
\end{example}

Let $S = \Spec k$. 

\begin{example} \label{good_ex_p1} \label{basic_example}
$\phi: [\AA^1 / \GG_m] \arr \Spec k$ is a good moduli space.  Similarly, $\phi: [\AA^2 / \GG_m] \arr \Spec k$ is a good moduli space.  The open substack $[\AA^2 \setminus \{0\} / \GG_m]$ is isomorphic to $\PP^1$.  This example illustrates that good moduli spaces may vary greatly as one varies the open substack.
\end{example}

\begin{example}
If $G$ is a linearly reductive group scheme over $k$ (see Section \ref{sect_lin_red}) acting a scheme $X = \Spec A$, then $\phi: [X / G] \arr \Spec A^G$ is a good moduli space (see Theorem \ref{git_affine_thm}).
\end{example}

\begin{example}
$\phi: [\PP^1 / \GG_m] \arr k$ is not a good moduli space.  Although condition (ii) of the definition is satisfied, $\phi$ is not cohomologically affine.  There are two closed points in $[\PP^1 / \GG_m]$ which have the same image under $\phi$ contradicting property (\ref{good_separate}) of Theorem \ref{good_thm}.
\end{example}

\begin{example}
$\phi: [\PP^1 / PGL_2] \arr \Spec k$ is not a good moduli space.  Indeed, there is an isomorphism of stacks $[\PP^1 / PGL_2] \cong B (\UT_2)$ where $\UT_2 \subset \GL_2$ is the subgroup of upper triangular matrices.  Since $\UT_2$ is not linearly reductive (see Section \ref{sect_lin_red}), $\phi$ is not cohomologically affine.
\end{example}

\begin{example}
We recall Mumford's example (\cite[Example 0.4]{git}) of a geometric quotient that is not universal for maps to algebraic spaces over $S = \Spec \CC$.  The example is:  $\SL_2$ acts naturally on the quasi-affine scheme
$$\begin{array}{ll}
X = \{ (L, Q^2) | & L \textrm{ is nonzero linear form}, \\ 
			& Q \textrm{ is a quadratic form with discriminant 1} \} 	
\end{array}$$
The action is set-theoretically free (ie. $\SL_2(k)$ acts freely on $X(k)$) but the action is not even proper (ie. $\SL_2 \times X \arr X \times X$ is not proper).  If we write $\cX = [X/\SL_2]$, then $\cX$ is the non-locally separated affine line which is an algebraic space but not a scheme.  The morphism
$$\begin{aligned}
\phi: 		X	& \arr \AA^1 \\
		(\alpha x + \beta y, Q^2)	& \mapsto Q^2(-\beta, \alpha)
\end{aligned}$$
is a geometric quotient.  Koll\'ar shows in \cite[Example 2.18]{kollar_quotients} that $\phi$ is not universal for maps to arbitrary algebraic spaces.  The induced map $\cX \arr \AA^1$ is not a good moduli space (as one can check directly that $\Gamma(\cX, \oh_{\cX}) \arr \Gamma(\cX, \oh_{\cX}/\cI)$ is not surjective where $\cI$ defines a nilpotent thickening of the origin) but obviously the identity morphism $\cX \arr \cX$ is a good moduli space. 
\end{example}

In the following examples, let $S = \Spec k$ with $k$ an algebraically closed field of characteristic 0.  The characteristic 0 hypothesis is certainly necessarily while the algebraically closed assumption can presumably be removed.

\begin{example} {\it Moduli of semi-stable sheaves} \label{example_semistable_sheaves}\\
Let $X$ be a connected projective scheme over $k$.  Fix an ample line bundle $\oh_X(1)$ on $X$ and a polynomial $P \in \QQ[z]$.  For a coherent sheaf $E$ on $X$ of dimension $d$, the \emph{reduced Hilbert polynomial} $p(E,m) = P(E,m) / \alpha_d(E)$ where $P$ is the Hilbert polynomial of $E$ and $\alpha_d / d!$ is the leading term.  A coherent sheaf $E$ on $X$ of dimension $d$ is called \emph{semi-stable} (resp. \emph{stable}) if $E$ is pure and for any proper subsheaf $F \subset E$, $p(F) \leq p(E)$ (resp. $p(F) < p(E)$).  A \emph{family of semi-stable sheaves over $T$ with Hilbert polynomial P} is a coherent sheaf $\cE$ on $X \times_S T$ flat over $T$ such that for all geometric points $t: \Spec K \arr T$, $\cE_t$ is semi-stable on $X_t$ with Hilbert polynomial $P$.

Let $\cM_{X,P}^{\ss}$ be the stack whose objects over $T$ are families of semi-stable sheaves over $T$ with Hilbert polynomial $P$ and a morphism from $\cE_1$ on $X \times_S T_1$ to $\cE_2$ on $X \times_S T_2$ is the data of a morphism $g: T_1 \arr T_2$ and an isomorphism $\phi: \cE_1 \arr (\id \times g)^*\cE_2$.  $\cM_{X,P}^{\ss}$ is an Artin stack finite type over $k$.  Let $\cM_{X,P}^{\s} \subseteq \cM_{X,P}^{\ss}$ be the open substack consisting of families of stables sheaves.  While every pure sheaf of dimension $d$ has a unique Harder-Narasimhan filtration where the factors are semi-stable, every semi-stable sheaf $E$ has a Jordan-H\"older filtration $0 = E_0 \subset E_1 \subset \cdots \subset E_l = E$ where the factors $\gr_i = E_i /E_{i-1}$ are stable with reduced Hilbert polynomial $p(E)$.  The graded object $\gr(E) = \bigoplus_i \gr_i(E)$ does not depend on the choice of Jordan-H\"older filtration.  Two semi-stable sheaves $E_1$ and $E_2$ with the same reduced Hilbert polynomial are called \emph{S-equivalent} if $\gr(E_1) \cong \gr(E_2)$.  A semi-stable sheaf is \emph{polystable} if can be written as the direct sum of stable sheaves.

The family of semi-stable sheaves on $X$ with Hilbert polynomial $P$ is bounded (see \cite[Theorem 3.3.7]{huybrechts-lehn}).  Therefore, there is an integer $m$ such that for any semi-stable sheaf $F$ with Hilbert polynomial $P$, $F(m)$ is globally generated and $h^0(F(m)) = P(m)$.  There is surjection $\oh_X(-m)^{P(m)} \arr F$ which depends on a choice of basis of $\Gamma(X,F(m))$.  There is an open subscheme $U$ of the Quot scheme $\Quot_{X,P}(\oh_X(-m)^{P(m)})$ parameterizing semi-stable sheaves and inducing an isomorphism on $H^0$ which is invariant under the natural action of $\GL_{P(m)}$ on $\Quot_{X,P}(\oh_X(-m)^{P(m)})$.   One can show that $\cM_{X,P}^{\ss} = [U/\GL_{P(m)}]$.  The arguments given by Gieseker and Maruyama and also later by Simpson (see \cite[Ch. 4]{huybrechts-lehn}) imply that there is a good moduli space $\phi: \cM_{X,P}^{\ss} \arr M_{X,P}^{\ss}$ where $M_{X,P}^{\ss}$ is projective.  Moreover, there is an open subscheme $M_{X,P}^{\s}$ such that $\phi^{-1}(M_{X,P}^{\s}) = \cM_{X,P}^{\s}$ and $\phi|_{\cM_{X,P}^{\s}}$ is a tame moduli space.  To summarize, we have
$$\xymatrix{
\cM_{X,P}^{\s} \ar@{^(->}[r] \ar[d]		& \cM_{X,P}^{\ss} \ar[d]^{\phi} \\
M_{X,P}^{\s} \ar@{^(->}[r]			& M_{X,P}^{\ss}
}$$
We stress that $\phi$ is not a coarse moduli space and two $k$-valued points of $\cM_{X,P}^{\ss}$ have the same image under $\phi$ if and only if the corresponding semi-stable sheaves are $S$-equivalent.
\end{example}

\begin{example} {\it Compactification of the universal Picard variety}\\
Let $g \ge 2$.  Recall that a \emph{semi-stable} (resp. \emph{stable}) \emph{curve of genus $g$ over $T$} is a proper, flat morphism $\pi: C \arr T$ whose geometric fibers are reduced, connected, nodal 1-dimensional schemes $C_t$ with arithmetic genus $g$ such that any non-singular rational component meets the other components in at least two (resp. three) points.  For a semi-stable curve $C \arr \Spec k$, the non-singular rational components meeting other components at precisely two points are called \emph{exceptional}. A \emph{quasi-stable curve of genus $g$ over $T$} is a semi-stable curve such that in any geometric fiber, no two exceptional components meet.  A line bundle $L$ of degree $d$ on a semi-stable curve $C \arr \Spec k$ of genus $g$ is said to be \emph{semi-stable} (or \emph{balanced}) if for every exceptional component $E$ of $C$, $\deg_E L = 1$, and if for every connected projective sub-curve $Y$ of genus $g_Y$ meeting the complement in $k_Y$ points, the degree $d_Y$ of $Y$ satisfies:
$$\big| d_Y - \frac{d}{g-1} (g_Y-1+k_Y/2) \big| \leq k_Y/2$$
It is shown in \cite{melo} that the stack $\bar{\cG}_{d,g}$ parameterizing quasi-stable curves of genus $g$ with semi-stable line bundles of degree $d$ is Artin.  There is an open substack $\cG_{d,g} \subseteq \bar{\cG}_{d,g}$ consisting of stable curves and the morphism $\cG_{d,g} \arr \cM_g$ is the \emph{universal Picard variety}.  Lucia 
Caporaso in \cite{caporaso} showed that there exists a good moduli space $\phi: \bar{\cG}_{d,g} \arr \bar{P}_{d,g}$ (which is not a coarse moduli space) where $\bar{P}_{d,g}$ is a projective scheme which maps onto $\bar{M}_g$.  Furthermore, there is an open subscheme $P_{d,g} \subseteq \bar{P}_{d,g}$ such that $\phi^{-1}(P_{d,g}) = \cG_{d,g}$ and $\phi|_{\cG_{d,g}}$ is a coarse moduli space.
\end{example}

\begin{example}
In \cite{schubert}, Schubert introduced an alternative compactification of $\cM_g$ parameterizing pseudo-stable curves.  A \emph{pseudo-stable} curve of genus $g$ is a connected, reduced curve with at worst nodes and cusps as singularities where every subcurve of genus 1 (resp. 0) meets the rest of the curve at least 2 (resp. 3) points.  For $g \ge 3$,the stack $\bar{\cM}_g^{\ps}$ of pseudo-stable curves is a separated, Deligne-Mumford stack admitting a coarse moduli space $\bar{M}_g^{\ps}$.  For $g = 2$, $\bar{\cM}_g^{\ps}$ is a non-separated Artin stack and admits a good moduli space $\phi: \bar{\cM}_g^{\ps} \arr \bar{M}_g^{\ps}$ which identifies all cuspidal curves (cuspidal curves whose normalization are elliptic curves, the cuspidal nodal curve whose normalization is $\PP^1$, and the bicuspidal curve whose normalization is $\PP^1$) to a point (see \cite{hassett_genus2}, \cite{hyeon-lee_genus2}).  The bicuspidal curve is the unique closed point in the fiber and has as stabilizer the the linearly reductive group $\GG_m \rtimes \ZZ_2$.  For $g \ge 2$, the schemes $\bar{M}_g^{\ps}$ are isomorphic to the log-canonical models $\bar{M}_g(\alpha) = \Proj \bigoplus_d(\bar{\cM}_g, d (K_{\bar{\cM_g}} + \alpha \Delta)$ for $7/10 < \alpha \leq 9/11$, where $\delta$ is the boundary divisor, and the morphism $\bar{M}_g \arr \bar{M}^{\ps}_g$ contracts $\Delta_1$, the locus of elliptic tails (see \cite{hassett-hyeon_contraction}).

Hassett and Hyeon show in \cite{hassett-hyeon_flip} for $g \ge 4$ (the $g=3$ case is handled in \cite{hyeon-lee_genus3}) that a flip occurs at the next step in the log minimal model program at $\alpha = 7/10$. Furthermore, they give modular interpretations for $\bar{M}_g(7/10)$ and $\bar{M}_g (7/10 -\epsilon)$ as the good moduli spaces (but not coarse moduli spaces) for the stack of \emph{Chow semi-stable curves} (where curves are allowed as singularities nodes, cusps, and tacnodes do not admit elliptic tails) and \emph{Hilbert semi-stable curves} (which are Chow semi-stable curves not admitting elliptic bridges), respectively.

\end{example}

\section{The topology of stacks admitting good moduli spaces} \label{topology_section}

\begin{prop}
Let $\cX$ be a locally noetherian Artin stack and $\phi: \cX \arr Y$ a good moduli space.  Given a closed point $y \in |Y|$, there is a unique closed point $x \in |\phi^{-1}(y)|$.  The dimension of the stabilizer of $x$ is strictly larger than the dimension of any other stabilizer in $\phi^{-1}(y)$.
\end{prop}

\bpf  The first statement follows directly from the fact that $\cX_y \arr \Spec k(y)$ is a good moduli space and therefore separates closed disjoint substacks.  Let $r$ be maximal among the dimensions of the stabilizers of points of $\phi^{-1} (y)$.  By upper semi-continuity (\cite[IV.13.1.3]{ega}),
$\cZ = \{z \in |\phi^{-1} (y)| \, | \,\dim G_z = r\} \subset \phi^{-1}(y)$ is a closed substack (given the reduced induced stack structure).  Let $x \in |\cZ|$ be a closed point.  If $\phi^{-1}(y) \setminus \{x\}$ is non-empty, there exists a point $x'$ closed in the complement.  Since there is an induced closed immersion $\cG_{x} \hookarr \overline{\cG_{x'}}$, $\dim \cG_{x} < \dim \cG_{x'}$ contradicting $\dim G_x = \dim G_{x'}$.  \epf

This unique closed point has linearly reductive stabilizer (see Proposition \ref{closed_orbits_lin_red}).  Conversely, it is natural to ask when a point of an Artin stack $\cX$ is in the closure of another point with lower dimensional stabilizer.  This question was motivated by discussions with Jason Starr and Ravi Vakil.  If $\cX$ admits a good moduli space, then the answer has a satisfactory answer:

\begin{prop}  Suppose $\cX$ is a noetherian Artin stack finite type over $S$ and $\phi: \cX \arr Y$ is a good moduli space.  Let $d$ be minimal among the dimensions of stabilizers of points of $\cX$.  Assume that the open substack $\cU = \{x \in |\cX| \, | \, \dim G_x = d\}$ is dense (for instance, if $\cX$ is irreducible).  Then any closed point $z \in |\cX|$ is in the closure of a point in $\cU$. 
\end{prop}

\bpf Define 
\begin{equation} 
\begin{aligned}
\sigma: |\cX| \arr \ZZ, \quad & x \mapsto \dim G_x\\
\tau: |\cX| \arr \ZZ, \quad & x \mapsto \dim_x \phi^{-1}(\phi(x))
\end{aligned}
\end{equation}

By applying  \cite[IV.13.1.3]{ega}, $\sigma$ is  upper semi-continuous 
and since $\phi: \cX \arr Y$ is finite type, $\tau$ is also upper semi-continuous.  In particular, $\cU$ is an open substack.  

Suppose $z \in |\cX| \setminus |\cU|$ is a closed point not contained in the closure of any point in $\cU$.  In particular $z \notin \cU$ so $\dim G_z > d$.  Set $y = \phi(z)$.  There is an induced closed immersion $\cG_z \hookarr \phi^{-1}(y)$ and a diagram
$$\xymatrix{
\cG_z \ar@{^(->}[r] \ar[rd] 	& \phi^{-1}(y) \ar[r] \ar[d]		& \cX \ar[d]^{\phi} \\
				& \Spec k(y) \ar[r]			& Y
}$$
where both $\cG_z \arr \Spec k(y)$ and $\phi^{-1}(y) \arr \Spec k(y)$ are good moduli spaces.  We claim that $\cG_z \hookarr \phi^{-1}(y)$ is surjective.  If not, there would exist a locally closed point $w \in \phi^{-1}(y)$ distinct from $z$ but containing $z$ in its closure.  But since $|\cG_z|$ is a proper closed subset of $\bar{|\cG_w|}$, $\dim G_w < \dim G_z$ contradicting our assumptions on $z$.  Therefore $\dim \cG_z = \dim_z \phi^{-1}(\phi(z))$. 

For any $x \in |\cX|$, we will show that $\dim \cG_x \leq \dim_x \phi^{-1} \phi(x)$.  Let $\cZ =  \{z \in \phi^{-1} \phi(x) | \dim G_x \ge \dim G_z \}$ which is a closed substack (with the induced reduced stack structure) of $\phi^{-1}(\phi(x))$.  Let $x' \in |\cZ|$ be a closed point.  The composition of the closed immersions $\cG_{x'} \hookarr \cZ \hookarr \phi^{-1}(\phi(x))$ induces the inequalities $\dim G_x \leq \dim G_{x'} \leq \dim_x \phi^{-1} \phi(x)$.  

For any point $x \in |\cX|$,
$$0 = \dim \cG_x + \dim G_x \leq \dim_x \phi^{-1} (\phi(x)) + \dim G_x$$

Set $r = \dim G_z > d$.  Let $\cW \subseteq \cX$ be the open substack consisting of points $x \in |\cX|$ such that $\dim G_w \leq r$ and $\dim_w \phi^{-1}(\phi(w)) \leq -r$.  Since  $\dim_w \phi^{-1} (\phi(w)) + \dim G_w \ge 0$, it follows that for all $w \in |\cW|$, $\dim G_w = r$ and $\dim \phi^{-1}(\phi(w)) = -r$ which contradicts that $\cU \subseteq \cX$ is dense. \epf

\section{Characterization of vector bundles} \label{vector_bundle_section}

If $\phi: \cX \arr Y$ is a good moduli space and $\cG$ is a vector bundle on $Y$, then $\phi^* \cG$ is a vector bundle on $\cX$ with the property that the stabilizers act trivially on the fibers.  It is natural to ask when a vector bundle $\cF$ on $\cX$ descends to $Y$ (that is, when there exists a vector bundle $\cG$ on $Y$ such that $\phi^* \cG \cong \cF$).  In this section, we prove that if $\cX$ is locally noetherian, there is an equivalence of categories between vector bundles on $Y$ and vector bundles on $\cX$ with the property that at closed points the stabilizer acts trivially on the fiber.  This result provides a generalization of the corresponding statement for good GIT quotients proved by Knop, Kraft and Vust in \cite{kkv} and \cite{kraft}.  We thank Andrew Kresch for pointing out the following argument.

\begin{defn} A vector bundle $\cF$ on a locally noetherian Artin stack $\cX$ has \emph{trivial stabilizer action at closed points} if for all geometric points $x: \Spec k \arr \cX$ with closed image, the representation of $G_x$ on $\cF \tensor k$ is trivial.  
\end{defn}

\begin{remark} This is equivalent to requiring that for all closed points $\xi \in |\cX|$, inducing a closed immersion $i: \cG_{\xi} \hookarr \cX$, there is an isomorphism $i^*\cF \cong \oh_{\cG_{\xi}}^n$ for some $n$.
\end{remark} 

\begin{thm} \label{vector_bundles}  If $\phi: \cX \arr Y$ is a good moduli space with $\cX$ locally noetherian, the pullback functor $\phi^*$ induces an equivalence of categories between vector bundles on $Y$ and the full subcategory of vector bundles on $\cX$ with trivial stabilizer action at closed points.  The inverse is provided by the push-forward functor $\phi_*$.
\end{thm}

\bpf We will show that if $\cF$ is a vector bundle on $\cX$ with trivial stabilizer action at closed points, the adjunction morphism $\lambda: \phi^* \phi_* \cF \arr \cF$ is an isomorphism and $\phi_* \cF$ is locally free.  These statements imply the desired result since the adjunction morphism $\cG \arr \phi_* \phi^* \cG$ is an isomorphism for any quasi-coherent $\oh_Y$-module (see Proposition \ref{good_prop1}).  

We may assume that $Y = \Spec A$ and $\cF$ is locally free of rank $n$.  We begin by showing that $\lambda$ is surjective.  Let $\xi \in |\cX|$ be a closed point which induces a closed immersion $i: \cG_{\xi} \hookarr \cX$ defined by a sheaf of ideals $\cI$, a closed point $y=\phi(\xi) \in Y$, and a commutative diagram
$$\xymatrix{
\cG_{\xi} \ar[r]^i \ar[d]^{\phi'}	& \cX \ar[d]^{\phi} \\
\Spec k(y) \ar[r]^{j}			& Y
}$$
It suffices to show that $i^* \lambda$ is surjective for any such $\xi$.  First, the adjunction morphism $\alpha: j^*\phi_* \cF \arr \phi'_*i^*\cF$ is surjective.  Indeed, $j_*\alpha$ corresponds under the natural identifications to $\phi_* \cF / ( \phi_*\cI  \phi_* \cF) \arr \phi_* (\cF / \cI \cF) \cong \phi_* \cF / \phi_* (\cI \cF)$ which is surjective since $\phi_* \cI \phi_* \cF \subseteq \phi_*(\cI \cF)$.  Now $i^* \lambda$ is the composition
$$i^*\phi^*\phi_*\cF \cong \phi'^* j^* \phi_* \cF \stackrel{\phi'^* \alpha}{\mapsonto} \phi'^* \phi'_* i^* \cF \iso i^* \cF$$
where the last adjunction morphism is an isomorphism precisely because $\cF$ has trivial stabilizer action at closed points.  Therefore, $\lambda$ is surjective.

Since $Y$ is affine, $\bigoplus_{s \in \Gamma(\cX, \cF)} \arr \phi_* \cF$ is surjective and it follows that the composition $\bigoplus_{s \in \Gamma(\cX, \cF)} \oh_{\cX} \arr \phi^* \phi_* \cF \arr \cF$ is surjective.  Let $\xi \in |\cX|$ be a closed point.  There exists $n$ sections of $\Gamma(\cX, \cF)$ inducing $\beta: \oh_{\cX}^n \arr \cF$ such that $\xi \notin \Supp(\coker \beta)$.  Let $V = Y \setminus \phi( \Supp(\coker \beta))$ and $\cU = \phi^{-1}(V)$.  Then $\xi \in \cU$ and $\beta|_{\cU}: \oh_{\cU}^n \arr \cF|_{\cU}$ is surjective morphism of vector bundles of the same rank and therefore an isomorphism.  It follows that $\phi_* \beta|_V: \oh_V^n \arr \phi_* \cF|_{\cU}$ and $\lambda|_{\cU}:\phi^* \phi_* \cF|_{\cU} \arr \cF|_{\cU}$ are isomorphisms.  This shows both that $\lambda$ is an isomorphism and that $\phi_* \cF$ is a vector bundle. \epf

\begin{remark}  The corresponding statement for coherent sheaves is not true.  Let $k$ be a field with $\charr(k) \ne 2$ and let $\ZZ_2$ act on $\AA^1 = \Spec k[x]$ by $x \mapsto -x$.  Then $[\AA^1 / \ZZ_2] \mapsto \Spec k[x^2]$ is a good moduli space.  If $i: B \ZZ_2 \hookarr [\AA^1/\ZZ_2]$ is the closed immersion corresponding to the origin, then $i_* \oh_{B \ZZ_2}$ does not descend.
\end{remark}

\section{Stability} \label{stability_section}

Artin stacks do not in general admit good moduli spaces just as linearly reductive group actions on arbitrary schemes do not necessarily admit good quotients.  Mumford studied linearized line bundles as a means to parameterize open invariant subschemes that do admit quotients.  In this section, we study the analogue for Artin stacks.  Namely, a line bundle on an Artin stack determines a (semi-)stability condition.  The locus of semi-stable points will admit a good moduli space and will contain the stable locus which admits a tame moduli space.  In particular, we obtain an answer to \cite[Question 19.2.3]{lmb}.

Let $\cX$ be an Artin stack with $p: \cX \arr S$ quasi-compact and $\cL$ be a line bundle on $\cX$.
 
\begin{defn}(Analogue of \cite[Definition 1.7]{git})  Let $x: \Spec k \arr \cX$ be a geometric point with image $s \in S$.
\begin{enumeratea}
\item $x$ is \emph{pre-stable} if there exists an open substack $\cU \subseteq \cX$ containing $x$ which is cohomologically affine over $S$ and has closed orbits.
\item $x$ is \emph{semi-stable} with respect to $\cL$ if there is an open $U \subseteq S$ containing $s$ and a section $t \in \Gamma(p^{-1}(U), \cL^n)$ for some $n > 0$ such that $t(x) \ne 0$ and
$p^{-1}(U)_t \arr U$ is cohomologically affine.
\item $x$ is \emph{stable} with respect to $\cL$ if there is an open $U \subseteq S$ containing $s$ and a section $t \in \Gamma(p^{-1}(U), \cL^n)$ for some $n > 0$ such that $t(x) \ne 0$, $p^{-1}(U)_t \arr U$ is cohomologically affine, and $p^{-1}(U)_t$ has closed orbits.
\end{enumeratea}
We will denote $\cX^{\s}_{\pre}$, $\cX^{\ss}_{\cL}$, and $\cX^{\s}_{\cL}$ as the corresponding open substacks.
\end{defn}

\begin{remark}
If $S = \Spec A$ is affine, then $x$ is semi-stable with respect to $\cL$ if and only if there exists a section $t \in \Gamma(\cX,\cL^n)$ for some $n > 0$ such that $t(x) \ne 0$ and $\cX_t$ cohomologically affine.  See Proposition \ref{stable_equiv} for equivalences of stability.
\end{remark}

\begin{remark} 
The $\Gamma(\cX, \oh_{\cX})$-module $\bigoplus_{n \ge 0} \Gamma(\cX, \cL^n)$ is a graded ring and will be called the \emph{projective ring of invariants}.  More generally, the $\oh_{S}$-module $\bigoplus_{n \ge 0} p_* \cL^n$ is a quasi-coherent sheaf of graded rings and is called the \emph{projective sheaf of invariants}. 
\end{remark}

\begin{prop} (Analogue of \cite[Proposition 1.9]{git}) If $\cX$ is an Artin stack quasi-compact over $S$, there is a tame moduli space $\phi: \cX^{\s}_{\pre} \arr Y$, where $Y$ is a scheme.  Furthermore, if $\cU \subseteq \cX$ is an open substack such that $\cU \arr Z$ is a tame moduli space, then $\cU \subseteq \cX^{\s}_{\pre}$.
\end{prop}

\bpf This follows from Propositions \ref{good_is_coarse} and \ref{almost_coarse_glue}. \epf

There is no guarantee that $\cX_{\pre}^{\s}$ is non-empty.  Furthermore, the scheme $Y$ in the preceding proposition may be very non-separated.  For instance, if $\cX = [(\PP^1)^4 / \PGL_2]$, $\cX_{\pre}$ is the open substack consisting of tuples of points such that three are distinct.  There is a good moduli space $\cX_{\pre} \arr Y$ where $Y$ is the non-separated projective line with three double points.

\begin{thm} (Analogue of \cite[Theorem 1.10]{git}) \label{general_good_thm}
Let $p: \cX \arr S$ be quasi-compact with $\cX$ an Artin stack and $\cL$ be a line bundle on $\cX$.  Then
\begin{enumeratei}
\item There is a good moduli space $\phi: \cX^{\ss}_{\cL} \arr  Y$ with $Y$ an open subscheme of $\sProj \bigoplus_{n \ge 0} p_* \cL^n$ and there is an open subscheme $V \subseteq Y$ such that $\phi^{-1}(V) = \cX^{\s}_{\cL}$ and $\phi |_{\cX^{\s}_{\cL}}: \cX^{\s}_{\cL} \arr V$ is a tame moduli space. 
\item If $\cX^{\ss}_{\cL}$ and $S$ are quasi-compact, then there exists an $S$-ample line bundle $\cM$ on $Y$ such that $\phi^* \cM \cong \cL^N$ for some $N >0$.  
\item If $S$ is an excellent quasi-compact scheme and $\cX$ is finite type over $S$, then $Y \arr S$ is quasi-projective.  

\end{enumeratei}
\end{thm}

\bpf By the universal property of sheafy proj, there exists a morphism $\phi: \cX^{\ss} \arr \sProj \bigoplus_{n \ge 0} p_* \cL^n$.  The set-theoretic image $Y$ is open and by the definition of semi-stability, $\phi: \cX^{\ss} \arr Y$ is Zariski-locally a good moduli space.  Let $V \subseteq Y$ be the union of open sets of the form $(\Proj \bigoplus_{n \ge 0} \Gamma(p^{-1}(U), \cL^n))_t$ where $U \subseteq S$ is affine, $t \in \Gamma(p^{-1}(U), \cL^n)$ for some $n > 0$ such that $p^{-1}(U)_t$ is cohomologically affine and has closed orbits.  It is clear that $\phi^{-1}(V) = \cX^{\s}_{\cL}$ and  Proposition \ref{good_is_coarse} implies $\phi |_{\cX^{\s}_{\cL}}: \cX^{\s}_{\cL} \arr V$ is a tame moduli space.  

The quasi-compactness of $\cX_{\cL}^{\ss}$ and $S$ implies that $Y$ is quasi-compact and there exists $N > 0$, a finite affine cover $\{S_i\}$ of $S$, and finitely many sections $t_{ij} \in \Gamma(p^{-1}(S_i), \cL^N)$ such that $Y$ is the union of open affines of the form $(\Proj \bigoplus_{n \ge 0} \Gamma(p^{-1}(S_i), \cL^n))_{t_{ij}}$.  It follows that $\cM = \oh(N)$ on $\sProj \bigoplus_{n \ge 0} p_* \cL^n$ is an $S$-ample line bundle and there is a canonical isomorphism $\phi^* \cM|_{Y} \cong \cL^N|_{\cX_{\cL}^{\ss}}$.

If in addition $S$ is excellent and $\cX$ is finite type, then Theorem \ref{good_thm}(\ref{good_finite_type}) implies that $Y \arr S$ is quasi-projective.  \epf

\begin{cor}  \label{coh_proj_equiv}
Let $\cX$ be an Artin stack finite type over $S$.  If $\cX$ admits a good moduli space projective over $S$ then $\cX \arr S$ is cohomologically projective.  If $S$ is excellent, the converse holds. \end{cor}

\bpf  Suppose $\phi: \cX \arr Y$ is a good moduli space with $Y$ projective over $S$.  Let $\cM$ be an ample line bundle on $Y$.  It is easy to see that $\phi^* \cM$ is cohomologically ample and since $\phi$ is universally closed, it follows that $\cX$ is cohomologically projective over $S$.  For the converse, there exists an $S$-cohomologically ample line bundle $\cL$ such that $\cX_{\cL}^{\ss} = \cX$ and $Y \arr S$ is quasi-projective.  Since $Y \arr S$ is also universally closed, the result follows. \epf

\begin{example} Over $\Spec \QQ$, the moduli stack, $\bar{\cM}_g$, of stable genus $g$ curves  and the moduli stack, $\cM_{X,P}^{\ss}$, of semi-stable sheaves on a connected projective scheme $X$ with Hilbert polynomial $P$, are cohomologically projective. \end{example}

\subsection{Equivalences for stability}

Suppose $\cX$ is a locally noetherian Artin stack and $\phi: \cX \arr Y$ is a good moduli space.  Recall the upper semi-continuous functions: 
\begin{equation} 
\begin{aligned}
\sigma: |\cX| \arr \ZZ, \quad & x \mapsto \dim G_x\\
\tau: |\cX| \arr \ZZ, \quad & x \mapsto \dim_x \phi^{-1}(\phi(x))
\end{aligned}
\end{equation}

If in addition $\phi: \cX \arr Y$ is a tame moduli space, then for all geometric points $x$, $\dim_x\phi^{-1} (\phi(x)) =  \dim BG_x$ by Proposition \ref{coarse_square}, which implies that
$$\sigma + \tau = 0.$$
so that $\sigma$ and $\tau$ are locally constant.  

\begin{defn}  $x \in |\cX|$ is \emph{regular} if $\sigma$ is constant in a neighborhood of $x$.  Denote $\cX^{\reg}$ the open substack consisting of regular points.
\end{defn}

\begin{lemma} \label{const_stab_closed_orbits}
 If $\cX$ is locally noetherian and $\sigma$ is locally constant in the geometric fibers of $S$, then $\cX$ has closed orbits.  In particular if $\cX = \cX^{\reg}$, $\cX$ has closed orbits.
\end{lemma}

\bpf  It suffices to consider $S = \Spec k$ with $k$ algebraically closed.  Suppose $x: \Spec \Omega \arr \cX$ is a geometric point such that $BG_x \arr \cX \times_k \Omega$ is not a closed immersion.  Since the dimension of the stabilizers of points of $\cX \times_k \Omega$ is also locally constant, we may assume $\Omega = k$.  The morphism $BG_x \arr \cX$ is locally closed so it factors as $BG_x \arr \cZ \arr \cX$, an open immersion followed by a closed immersion.  Let $y$ be a $k$-valued point in $\cZ$ with closed orbit.  Since $\cZ$ is irreducible (as $BG_x$ is irreducible), $\dim BG_y < \dim \cZ$ but $\dim BG_x = \dim \cZ$.  It follows that $\sigma$ is not locally constant at $y$.
\epf

\begin{prop}(Analogue of \cite[Amplification 1.11]{git})  \label{stable_equiv}
Let $\cX$ be a noetherian Artin stack which is finite type over an affine scheme $S$ and $\cL$ a line bundle on $\cX$.  Let $x$ be a geometric point of $\cX^{\ss}_{\cL}$.  Then the following are equivalent:
\begin{enumeratei}
\item $x$ is a point of $\cX^{\s}_{\cL}$.
\item $x$ is regular and has closed orbit in $\cX^{\ss}_{\cL}$
\item $x$ is regular and there is a section $t \in \Gamma(\cX, \cL^N)$ for $N > 0$ with $t(x) \ne 0$ and such that  $\cX_t$ is cohomologically affine and $x$ has closed orbit in $\cX_t$.
\end{enumeratei}
\end{prop}

\bpf   We begin with showing that (i) implies (ii).  Let $\phi: \cX^{\ss}_{\cL} \arr Y$ be a good moduli space and $V \subseteq Y$ such that $\phi^{-1}(V) = \cX^{\s}_{\cL}$.  Write $x: \Spec k \arr \cX$ and let $\bar{\cX} = \cX \times_S k$, $\bar{Y} = Y \times_S k$,...  Consider
$$\xymatrix{
BG_x \ar[r] 	& \bar \cX^{\s}_{\cL} \times_{\bar V} \Spec k \ar[r] \ar[d]				& \bar \cX^{\s}_{\cL} \ar[r] \ar[d]		&\bar \cX^{\ss}_{\cL} \ar[d] \\
			& \Spec k \ar[r]^{\bar \phi (\bar x)}		& \bar V \ar[r]					& \bar Y &\quad 
}$$
First, all points in $\cX^{\s}_{\cL}$ are regular. By Proposition \ref{coarse_square}, the composition $BG_x \arr  \bar \cX^{\s}_{\cL} \times_{\bar V} \Spec k \arr  \bar \cX^{\s}_{\cL}$ is a closed immersion.  

It clear the (ii) implies (iii).  Suppose (iii) is true and define the closed substacks of $\cX_t$ by $\cS_r = \{x \in |\cX_t| \, \big| \dim G_x \ge r\}$.  For some $r$, $x \in \cS_r \setminus \cS_{r+1}$.  If we let
$$\begin{aligned}
\cZ_1	&= \{ x\} \\
\cZ_2 &= \cS_{r+1} \cup \overline{ \cX_t \setminus \cS_r }
\end{aligned}$$
which are closed substacks of $\cX_t$.  Since $x$ is regular, they are disjoint.  We have $\phi: \cX_t \arr \Spec \Gamma(\cX_t, \oh_{\cX})$ is a good moduli space and by Proposition \ref{good_thm}(\ref{good_separate}), $\phi(Z_1) \cap \phi(Z_2) = \emptyset$.  There exists $f \in \Gamma(\cX_t, \oh_{\cX})$ with $f(x) \ne 0$ and $f|_{\cZ_2} = 0$.  The stabilizers of points in $(\cX_t)_f$ have the same dimension so by Lemma \ref{const_stab_closed_orbits}, $(\cX_t)_f$ has closed orbits.  Finally, since $\cX_s$ is quasi-compact, there exists an $M$ such that $t^M \cdot f \in \Gamma(\cX, \cL^{MN})$ and $(\cX_t)_f = \cX_{t^M \cdot f}$.  This implies (i).   \epf

\subsection{Converse statements}

The semi-stable locus of a line bundle admits a quasi-projective good moduli space.  In this section, we show the converse holds under suitable hypotheses: given an open substack which admits a quasi-projective good moduli space, then the open substack is contained in the semi-stable locus of some line bundle.  The following theorem provides a generalization of \cite[Converse 1.13]{git}.  We note that although Mumford states the result for the stable locus, the same proof holds for the semi-stable locus.  

We thank Angelo Vistoli for pointing out the proof of the following lemma:

\begin{lemma} Let $\cX$ be a noetherian regular Artin stack and $\cU \subseteq \cX$ an open substack.  Then $\Pic(\cX) \arr \Pic(\cU)$ is surjective.
\end{lemma}

\bpf Let $i: \cU \hookarr \cX$.  If $L$ is a line bundle on $\cU$, then by \cite[Cor. 15.5]{lmb}, there exists a coherent sheaf $\cF$ such that $\cF|_{\cU} \cong \cL$.  Since $\cF^{\v\v}|_{\cU} \cong \cL^{\v\v} \cong \cL$, we may assume that $\cF$ is reflexive.  Since $\cX$ is noetherian and regular, any reflexive rank 1 sheaf is invertible. 
\epf

\begin{theorem}  (Analogue of \cite[Converse 1.13]{git}) Let $\cX$ be a noetherian regular Artin stack with affine diagonal over a quasi-compact scheme $S$.  Then
\begin{enumeratei}
\item If $\cW \subseteq \cX$ is an open substack and $\varphi: \cW \arr W$ is a tame moduli space with $W$ quasi-projective over $S$, then there exists a line bundle $\cL$ on $\cX$ such that $\cW \subseteq \cX^{\s}_{\cL}$.
\item If $\cU \subseteq \cX$ is an open substack and $\psi: \cU \arr U$ is a good moduli space with $U$ quasi-projective over $S$, then there exists a line bundle $\cL$ on $\cX$ such that $\cU \subseteq \cX^{\ss}_{\cL}$ and $\cU$ is saturated for the good moduli space $\phi: \cX^{\ss}_{\cL} \arr Y$.
\item If $\cW,\cU, \varphi, \psi$ are as in (i) and (ii) such that $W \subseteq U$ and $\cW = \psi^{-1}(W)$, then there exists a line bundle $\cL$ on $\cX$ such that $\cU \subseteq \cX^{\ss}_{\cL}$.  In particular, we have a diagram
$${\def\objectstyle{\scriptstyle}
\def\labelstyle{\scriptstyle}
\xymatrix@=20pt{
				&\cX^{\s}_{\cL} \ar@{^(->}[rr] \ar[dd] 			&				& \cX^{\ss}_{\cL} \ar[dd]^{\phi}  \\
\cW \ar@{^(->}[rr] \ar[dd]^{\varphi} \ar@{^(->}[ur] &						& \cU \ar[dd]^{\psi} \ar@{^(->}[ur]	& \\ 
				&V \ar@{^(->}[rr] 			& 				& Y  \\
W\ar@{^(->}[rr] \ar@{^(->}[ur]	&							&U \ar@{^(->}[ur]		&
}}$$
where the four vertical faces are cartesian and the far square is as in Theorem \ref{general_good_thm}.

\end{enumeratei}
\end{theorem}

\bpf For (ii), let $\cM$ be an $S$-ample line bundle on $U$.  By the lemma, there exists a line bundle $\cL$ on $\cX$ extending $\psi^* \cM$.  Let $\cD_1, \ldots \cD_k$ be the components of $\cX \setminus \cU$ of codimension 1 and write $\cL_{N} = \cL \tensor \oh_{\cX}( N (\sum_i \cD_i))$.

Set $q: U \arr S$ and $p: \cX \arr S$. Let $S' \subseteq S$ be an affine open and set $U' = q^{-1}(U)$, $\cX' = p^{-1}(U')$ and $\cU' = \cX' \cap \cU$.  We will show that if $s_0 \in \Gamma(U', \cM^{\tensor n})$ such that $U'_{s_0}$ is affine, then $s=\psi^*s_0$ extends to a section $t \in \Gamma(\cX', \cL_N^{\tensor n})$ for some $N$ such that $\cX'_t = \cU'_{s}$.  We may choose $N$ large enough such that $s$ extends to a section $t$ which vanishes on each $\cD'_i = \cD_i \cap \cX'$.  We have
$$\cU'_s \subseteq \cX'_t \subseteq \cX' \setminus \cup_i \cD'_i$$ 
If $g: X \arr \cX'_t$ is a smooth presentation with $X$ a scheme, then since $\cX \arr S$ has affine diagonal, $g$ is an affine morphism.  Since $\cU'_s$ is cohomologically affine, $U = g^{-1}(\cU'_s)$ is an affine scheme and therefore all components of $X \setminus U$ have codimension 1.  Since $t$ vanishes on each codimension 1 component of $\cX' \setminus \cU'$, it follows that $\cX'_t = \cU'_s$.  Therefore, $\cU'_s \subseteq \cX_{\cL_N}^{\ss}$.  We may cover $\cU$ with finitely many such open substacks.  Clearly, $U \subseteq Y$ and $\cU = \phi^{-1}(U)$.

Statements (i) and (iii) follow from similar arguments by realizing that the open substacks $\cU'_s$ have closed orbits by Proposition \ref{good_is_coarse}. \epf

\section{Linearly reductive group schemes} \label{sect_lin_red}

\begin{defn} An fppf group scheme $G \arr S$ is \emph{linearly reductive} if the morphism $BG \arr S$ is cohomologically affine. \end{defn}

\begin{remark}  Clearly $G \arr S$ is linearly reductive if and only if $BG \arr S$ is a good moduli space. \end{remark}

\begin{remark} If $S = \Spec k$, this is equivalent to usual definition of linearly reductive (see Proposition \ref{lin_red_equiv_k}).   If $\charr k = 0$, then $G \arr \Spec k$ is linearly reductive if and only if $G \arr \Spec k$ is reductive (ie. the radical of $G$ is a torus).  

Linear reductive finite flat group schemes of finite presentation have been classified recently by Abramovich, Olsson and Vistoli in \cite{tame}.  Over a field, linearly reductive algebraic groups have been classified by Nagata in \cite{nagata_complete}.  It is natural to ask whether these results can be extended to arbitrary linearly reductive group schemes.

If $G \arr S$ is a finite flat group schemes of finite presentation, then $G \arr S$ is linearly reductive if and only if the geometric fibers are linearly reductive (\cite[Theorem 2.19]{tame}).  If in addition $S$ is noetherian, linearly reductivity can even be checked on the fibers of closed points of $S$.

This result does not generalize to arbitrary fppf group schemes $G \arr S$.  Indeed, if $S = \ZZ[\frac{1}{2}]$, let $G \arr \AA^1$ be the group scheme with fibers $\ZZ / 2 \ZZ$ over all points except over the origin where the fiber is the trivial.  There is a unique non-trivial action of $G$ on $\AA^2 \arr \AA^1$.  Let $\cX=[\AA^1 / G]$ and $\cX_0$ be the fiber over the origin.  Then $\Gamma(\cX, \oh_{\cX}) \arr \Gamma(\cX_0, \oh_{\cX_0})$ is not surjective (ie. invariants can't be lifted) implying $G \arr \AA^1$ is not linearly reductive.  Clearly the geometric fibers are linearly reductive.  One might hope that if $G \arr S$ has geometrically connected fibers, then linearly reductivity can be checked on geometric fibers.

If $G \arr S$ is an fppf group scheme, it is not an open condition on $S$ that the fibers are linearly reductive.  For example, the only fiber of $GL_n(\ZZ) \arr \Spec \ZZ$ which is linearly reductive is the generic fiber.  If in addition $G \arr S$ is finite, then by Proposition \cite[Lemma 2.16 and Theorem 2.19]{tame}, this is a local property.
\end{remark}

\begin{example} \label{lin_red_examples} \quad 
\begin{enumerate1}
\item $\GL_n$, $\PGL_n$ and $\SL_n$ are linearly reductive over $\QQ$.  They are not linearly reductive over $\ZZ$ although $\GL_n$ and $\PGL_n$ are \emph{reductive} group schemes over $\ZZ$.
\item  A diagonalizable group scheme is linearly reductive (\cite[I.5.3.3]{sga3}). In particular, any torus $(\GG_m)^n \arr S$ is linearly reductive and $\mu_n \arr S$ is linearly reductive where $\mu_n =  \Spec \ZZ[t]/(t^n-1) \times_{\ZZ} S$ .
\item An abelian scheme (ie. smooth, proper group scheme with geometrically connected fibers) is linearly reductive.
\end{enumerate1}
\end{example}

\begin{prop}  \label{lin_red_noeth} (Generalization of \cite[Proposition 2.5]{tame})
Suppose $S$ is noetherian and $G \arr S$ be an fppf group scheme.  The following are equivalent:
\begin{enumeratei}
\item $G \arr S$ is linearly reductive.
\item The functor $\coh^G(S) \arr \coh(S)$ defined by $F \mapsto F^G$ is exact.
\end{enumeratei}
\end{prop}

\bpf This is clear from Proposition \ref{coh_aff_coherence}. \epf

\begin{prop}  \label{lin_red_equiv_k}
Let $G \arr \Spec k$ be a finite type and separated group scheme.  The following are equivalent:
\begin{enumeratei}
\item $G$ is linearly reductive.
\item The functor $V \mapsto V^G$ from $G$-representations to vector spaces is exact.
\item The functor $V \mapsto V^G$ from finite dimensional $G$-representations to vector spaces is exact.
\item Every $G$-representation is completely reducible.
\item Every finite dimensional $G$-representation is completely reducible.  
\item For every finite dimensional $G$-representation $V$ and $0 \ne v \in V^G$, there exists $F \in (V^{\vee})^G$ such that $F(v) \ne 0$.
\end{enumeratei}
\end{prop}

\bpf  The category of quasi-coherent $\oh_{BG}$-modules is equivalent to category of $G$-representations so that (ii) is a restatement of the definition of linearly reductive.  Proposition \ref{lin_red_noeth} implies that (ii) is equivalent to (iii).   For (iii) $\implies$ (v), if $0 \arr V_1 \arr V_2 \arr V_3 \arr 0$ is an exact sequence of finite dimensional $G$-representations, then by applying the functor $\Hom^G(V_3, \cdot) = \Hom^G(k,V_3^\vee \tensor \cdot) = (V_3^{\vee} \tensor \cdot )^G$ which is exact, we see that the sequence splits.  Conversely, it is clear that (v) $\implies$ (iii).  A simple application of Zorn's lemma implies that (iv) $\iff$ (v).  We have established the equivalences of (i) through (v).

For (iii) $\implies$ (vi), $0 \neq v \in V^G$ gives a surjective morphism of $G$-representations $v: V^{\vee} \arr k, \alpha \mapsto \alpha(v)$.  After taking invariants, $(V^{\vee})^G \arr k$ is surjective which implies there exists $F \in (V^{\vee})^G$ with $F(v) \neq 0$.  Conversely for (vi) $\implies$ (iii), suppose $\alpha: V \arr W$ is a surjective morphism of finite dimensional $G$-representations and $w \in W^G$.  Then $\alpha^{-1}(w) = V' \arr k$ is surjective morphism of $G$-representations giving $0 \neq F \in V'^{\vee}$ so by (vi) there exists $v' \in V'^G \subseteq V^G$ with $F(v') \neq 0$.  The image of $v' \in W^G$ is a scalar multiple of $w$ so it follows that $V^G \arr W^G$ is surjective.
\epf

\begin{remark}
The equivalences of (ii) - (vi) remain true without the assumptions that $G$ is finite type and separated over $k$.
\end{remark}

\begin{prop}  (Generalization of \cite[Proposition 2.6]{tame})
Let $G \arr S$ be an fppf group scheme, $S' \arr S$ a morphism of schemes and $G' = G \times_S S'$.  Then
\begin{enumeratei}
\item If $G \arr S$ is linearly reductive, then $G' \arr S'$ is linearly reductive.
\item If $S' \arr S$ is faithfully flat and $G' \arr S'$ is linearly reductive, then $G \arr S$ is linearly reductive.
\end{enumeratei}
\end{prop}

\bpf  Since $BG' = BG \times_S S'$, this follows directly from Proposition \ref{coh_aff_prop}. \epf

\begin{example} \label{ex_lin_red}
If $G \arr S$ is a linearly reductive group scheme acting on a scheme $X$ affine over $S$, then $p: [X/G] \arr S$ is cohomologically affine.  Indeed, there is a 2-cartesian square:
$$\xymatrix{
X \ar[r] \ar[d]	& S \ar[d]\\
[X/G] \ar[r]		& BG
}$$
Since $S \arr BG$ is fppf and $X \arr S$ is affine, $[X/G] \arr BG$ is an affine morphism.  This implies that the composition $[X/G] \arr BG \arr S$ is cohomologically affine.  Furthermore, from the property P argument of of \ref{coh_aff_propP}, it follows that $[X/G] \arr p_* \oh_{[X/G]}$ is a good moduli space.

Conversely, if $G \arr S$ is an affine group scheme acting on an algebraic space  $X$ and $[X/G] \arr S$ is cohomologically affine, then $X$ is affine over $S$.  This follows from Serre's criterion (see Proposition \ref{serre_crit}) since $X \arr S$ is the composition of the affine morphism $X \arr [X/G]$ with the cohomologically affine morphism $[X/G] \arr S$. 
\end{example}

\begin{example}  A morphism of Artin stacks $f: \cX \arr \cY$ is said to have affine diagonal if $\Delta_{\cX/\cY}: \cX \arr \cX \times_{\cY} \cX$ is an affine morphism.  The property of a morphism having affine diagonal is stable under composition, arbitrary base change and satisfies fppf descent.   If $G \arr S$ is an fppf affine group scheme acting on an algebraic space $X \arr S$ with affine diagonal, then $[X/G] \arr S$ has affine diagonal.  Indeed, let $\cX = [X/G]$ and consider
$$\xymatrix{
G \times_S X \ar[r]^{\psi} \ar[d]		& X \times_S X \ar[r]^{p_1} \ar[d]	& X \\
\cX \ar[r]^{\Delta_{\cX/S}}			& \cX \times_S \cX
}$$
where the square is 2-cartesian.  Since $G \arr S$ is affine, $p_1 \circ \psi$ is affine.  Since $X \arr S$ has affine diagonal, $p_1$ has affine diagonal.  It follows from the property P argument of \ref{coh_aff_propP} that $\psi$ is affine so by descent $\cX \arr S$ has affine diagonal.  In particular, $BG \arr S$ has affine diagonal.  
\end{example}

\subsection{Linearly reductivity of stabilizers, subgroups, quotients and extensions}

\begin{prop} \label{lin_red_stab2}  Suppose $\cX$ is a locally noetherian Artin stack and $\xi \in |\cX|$.  If $x: \Spec k \arr \cX$ is any representative, then $G_x$ is linearly reductive if and only if $\cG_{\xi}$ is cohomologically affine.
\end{prop}

\bpf This follows from diagram \ref{BGgerbe} and fpqc descent. \epf

The above proposition justifies the following definition.

\begin{defn}  If $\cX$ is a locally noetherian Artin stack, a point $\xi \in |\cX|$ has a \emph{linearly reductive stabilizer} if for some (equivalently any) representative $x: \Spec k \arr \cX$, $G_x$ is linearly reductive.
\end{defn}

The following is an easy but useful fact insuring linearly reductivity of closed points.

\begin{prop} \label{closed_orbits_lin_red}
Let $\cX$ be a locally noetherian Artin stack and $\phi: \cX \arr Y$ a good moduli space.  Any closed point $\xi \in |\cX|$ has a linearly reductive stabilizer.  In particular, for every $y \in Y$, there is a $\xi \in |\cX_y|$ with linearly reductive stabilizer.
\end{prop}

\bpf The point $\xi$ induces a closed immersion $\cG_{\xi} \hookarr \cX$.  By Lemma \ref{good_affine_prop}, the morphism from $\cG_{\xi}$ to its scheme-theoretic image, which is necessarily $\Spec k(\xi)$, is a good moduli space.  Therefore $\xi$ has linearly reductive stabilizer.
\epf

\subsection*{Matsushima's Theorem} We can now give a short proof of an analogue of a result sometimes referred to as Matsushima's theorem (see \cite[Appendix 1D]{git3} and \cite{matsushima}):  If $H$ is a subgroup of a \emph{reductive} group scheme $G$, then $H$ is \emph{reductive} if and only if $G/H$ is affine. In \cite{matsushima}, Matsushima proved the statement over the complex numbers using algebraic topology.  The algebro-geometric proof in the characteristic zero case is due Bialynicki-Birula in  \cite{bb_homogeneous} and a characteristic $p$ generalization was provided by Haboush in \cite{haboush_stab} and Richardson in \cite{richardson}.

\begin{thm} \label{matsushima} \quad
Suppose $G \arr S$ is a linearly reductive group scheme and $H \subseteq G$ is an fppf subgroup scheme.  Then
\begin{enumeratei}
\item If $G/H \arr S$ is affine, then $H \arr S$ is linearly reductive.
\item Suppose $G \arr S$ is affine.  If $H \arr S$ is linearly reductive, then $G/H \arr S$ is affine. 
\end{enumeratei}
Suppose $\cX$ is a locally noetherian Artin stack and $\xi \in |\cX|$.  Then
\begin{enumeratei} \setcounter{enumi}{2}
\item If $\cX \arr S$ is cohomologically affine and $\cG_{\xi} \arr \cX$ is affine, then $\xi$ has a linearly reductive stabilizer.
\item If $\cX \arr S$ has affine diagonal and $\xi$ has a linearly reductive stabilizer, then $\cG_{\xi} \arr \cX$ is affine.
\end{enumeratei}
In particular, if $\cX=[X/G]$ where $G \arr S$ is an affine, linear reductive group scheme and $X \arr S$ is affine, then $\xi$ has a linearly reductive stabilizer if and only if $O_X(\xi) \arr X$ is affine.
\end{thm}

\bpf  For (i) and (ii), the quotient stack $[G/H]$ is an algebraic space which we will denote by $G/H$.  Since the square
$$\xymatrix{
G/H \ar[r] \ar[d]		& S \ar[d]\\
BH \ar[r]			& BG
}$$
is 2-cartesian, $BH \arr BG$ is affine if and only if $G/H \arr S$ is affine.  By considering the composition $BH \arr BG \arr S$, it is clear that if $G/H \arr S$ is affine, then $H$ is linearly reductive.  For the converse, since $BG \arr S$ has affine diagonal, the property P argument of \ref{coh_aff_propP} implies that $G/H \arr S$ is cohomologically affine and therefore affine by Serre's criterion (see Proposition \ref{serre_crit}).

For (iii) and (iv), consider the commutative square
$$\xymatrix{
\cG_{\xi} \ar[r] \ar[d]		& \cX \ar[d] \\
\Spec k(\xi) \ar[r]		& S
}$$
For (iii), the composition $\cG_{\xi} \arr \cX \arr S$ is cohomologically affine.  Since $\Spec k(\xi) \arr S$ has affine diagonal, $\cG_{\xi} \arr \Spec k(\xi)$ is cohomologically affine so $\xi$ has linearly reductive stabilizer.  For (iv), since $\xi$ has linearly reductive stabilizer, the composition $\cG_{\xi} \arr \Spec k(\xi) \arr S$ is cohomologically affine.  Because $\cX \arr S$ has affine diagonal, $\cG_{\xi} \arr \cX$ is cohomologically affine and therefore affine by Serre's criterion. \epf

More generally, we can consider the relationship between the orbits and stabilizers of $T$-valued points.

\begin{prop}  \label{lin_red_stab1}
Let $\cX \arr S$ be an Artin stack and $f: T \arr \cX$ be such that $G_f$ is an fppf group scheme over $T$. Then
\begin{enumeratei}
\item  If $\cX \arr S$ is cohomologically affine and the natural map $BG_f \arr \cX \times_S T$ is affine, then $G_f \arr T$ is linearly reductive.  
\item If $\cX \arr S$ has affine diagonal and $G_f \arr T$ is linearly reductive, then the natural map $BG_f \arr \cX \times_S T$ is affine.
\end{enumeratei}  
In particular, if $\cX=[X/G]$ where $G \arr S$ is linear reductive, $X \arr S$ is affine, and $f: T \arr X$ has fppf stabilizer $G_f \arr T$, then $G_f \arr T$ is linearly reductive if and only if $o_X(f) \hookrightarrow X \times_S T$ is affine.
\end{prop}

\bpf Consider the composition $BG_f \arr \cX \times_S T \arr T$.  The first part is clear and the second part follows from the property P argument of \ref{coh_aff_propP} and Serre's criterion.
\epf

Matsushima's theorem characterizes subgroup schemes of a linearly reductive group that are linearly reductive.  The following generalization of \cite[Proposition 2.7]{tame} shows that quotients and extensions of linearly reductive groups schemes are also linearly reductive.

\begin{prop} 
Consider an exact sequence of fppf group schemes
$$1 \arr G' \arr G \arr G'' \arr 1 $$
\begin{enumeratei}  
\item If $G \arr S$ is linearly reductive, then $G'' \arr S$ is linearly reductive.
\item If $G' \arr S$ and $G'' \arr S$ are linearly reductive, then $G \arr S$ is linearly reductive.
\end{enumeratei}
\end{prop}

\bpf 
We first note that for any morphism of fppf group schemes $G' \arr G$ induces a morphism $i: BG' \arr BG$ with $i^*$ exact.  Indeed $p: S \arr BG'$ and $i \circ p$ are faithfully flat and $i^*$ is exact since $p^* \circ i^*$ is exact.  There is an induced commutative diagram
$$\xymatrix{
BG' \ar[r]^i \ar[rd]_{\pi_{G'}	}	& BG \ar[r]^j \ar[d]^{\pi_G} 	& BG'' \ar[ld]^{\pi_{G''} }\\
						& S
}$$
and a 2-cartesian diagram
$$\xymatrix{
BG' \ar[r]^i \ar[d]^{\pi_{G'}}		& BG' \ar[d]^j \\
S \ar[r]^p					&BG''
}$$
The natural adjunction morphism $\id \arr j_* j^*$ is an isomorphism.  Indeed it suffices to check that $p^* \arr p^* j_* j^*$ is an isomorphism and there are canonical isomorphisms $p^* j_* j^* \cong {\pi_{G'}}_*i^* j^* \cong {\pi_{G'}}_* \pi_{G'}^* p^*$ such that the composition $p^* \arr {\pi_{G'}}_* \pi_{G'}^* p^*$ corresponds the composition of $p^*$ and the adjunction isomorphism $\id \arr {\pi_{G'}}_* \pi_{G'}^*$.

To prove (i), we have isomorphisms of functors
$${\pi_{G''}}_* \iso {\pi_{G''}}_* j_* j^* \cong {\pi_G}_* j^*$$
with ${\pi_G}_*$ and $j^*$ exact functors.

To prove (ii), $j$ is cohomologically affine since $p$ is faithfully flat and $G' \arr S$ is linearly reductive.   As $\pi_G = \pi_{G''} \circ j$ is the composition of cohomologically affine morphisms, $G \arr S$ is linearly reductive.
\epf

\section{Geometric Invariant Theory} \label{git_section}

The theory of good moduli space encapsulates the geometric invariant theory of linearly reductive group actions.  We rephrase some of the results from Section \ref{good_sect}-\ref{sect_lin_red} in the special case when $\cX$ is quotient stack by a linearly reductive group scheme.

\subsection{Affine Case}
Let $G \arr S$ be a linearly reductive group scheme acting an a scheme $p: X \arr S$ with $p$ affine.  
\begin{thm} \label{git_affine_thm} (Analogue of \cite[Theorem 1.1]{git})  The morphism
$$\phi: [X/G] \arr \sSpec p_* \oh_{[X/G]}$$
 is a good moduli space. 
\end{thm}

\bpf  This is immediate from Example \ref{ex_lin_red}. \epf

\begin{remark}  If $S = \Spec k$, $X = \Spec A$ and $G$ is a smooth affine linearly reductive group scheme, this is \cite[Theorem 1.1]{git} and 
$$X \arr \Spec A^G$$
is the GIT good quotient.
\end{remark}

\begin{cor}  \label{git_flat_families}
\emph{GIT quotients behave well in flat families.}  With the hypotheses of Theorem \ref{git_affine_thm}, for any field valued point $s: \Spec k \arr S$, the induced morphism $\phi_s: [X_s / G_s] \arr Y_s$ is a good moduli space with $Y_s \cong \Spec \Gamma(X_s, \oh_{X_s})^{G_s}$.  If $X \arr S$ is flat, then $Y \arr S$ is flat.
\end{cor}

\bpf If $X \arr S$ is flat, then $\cX = [X/G] \arr S$ is flat and by Theorem \ref{good_thm}(\ref{good_flat}), $Y  \arr S$ is flat.  The second statement follows since good moduli spaces are stable under arbitrary base change and $\cX_s \cong [X_s / G_s]$.
\epf

\subsection{General case} Let $G \arr S$ be a linearly reductive group scheme acting an a scheme $p: X \arr S$ with $p$ quasi-compact.  Suppose $L$ is a $G$-linearization on $X$.  Let $\cX = [X/G]$, $g: X \arr \cX$ and $\cL$ the corresponding line bundle on $\cX$.   Define $X^{\ss}_{L} = g^{-1}(\cX^{\ss}_{\cL})$ and $X^{\s}_{L} = g^{-1}(\cX^{\s}_{\cL})$.  If $S = \Spec k$, then this agrees with the definition of (semi-)stability in \cite[Definition 1.7]{git}.

\begin{thm} \label{git_general_thm} (Analogue of \cite[Theorem 1.10]{git})
\begin{enumeratei}
\item There is a good moduli space $\phi: \cX^{\ss}_{\cL} \arr  Y$ with $Y$ an open subscheme of $\sProj \bigoplus_{n \ge 0} (p_* \cL^n)^G$ and there is an open subscheme $V \subseteq Y$ such that $\phi^{-1}(V) = \cX^{\s}_{\cL}$ and $\phi |_{\cX^{\s}_{\cL}}: \cX^{\s}_{\cL} \arr V$ is a tame moduli space. 
\item If $X^{\ss}_L$ and $S$ are quasi-compact over $S$ (for example, if $|X|$ is a noetherian topological space), then there exists an $S$-ample line bundle $\cM$ on $Y$ such that $\phi^* \cM \cong \cL^N$ for some $N$.  
\item If $S$ is an excellent quasi-compact scheme and $X$ is finite type over $S$, then $Y \arr S$ is quasi-projective.  If $X \arr S$ is projective and $\cL$ is relatively ample, then $Y \arr S$ is projective.

\end{enumeratei}
\end{thm}

\bpf This is a direct translation of Theorem \ref{general_good_thm}. For the final statement, the extra hypotheses imply that for every section $s \in \Gamma(\cX^{\ss}, \cL^n)$ over an affine in $S$, the locus $\cX_s$ is cohomologically affine which implies that $Y = \sProj \bigoplus_{n \ge 0} (p_* \cL^n)^G$. 
\epf

\begin{remark} 
If $S = \Spec k$ and $G$ is a smooth affine linearly reductive group scheme, this is \cite[Theorem 1.10]{git} and 
$$X_{L}^{\ss} \arr Y \subseteq \Proj \bigoplus_{n \ge 0} \Gamma(X,L^n)^G$$
is the GIT good quotient.
\end{remark}

\newpage
\bibliography{../references}{}

\def\polhk#1{\setbox0=\hbox{#1}{\ooalign{\hidewidth
  \lower1.5ex\hbox{`}\hidewidth\crcr\unhbox0}}}
\providecommand{\bysame}{\leavevmode\hbox to3em{\hrulefill}\thinspace}
\providecommand{\MR}{\relax\ifhmode\unskip\space\fi MR }
\providecommand{\MRhref}[2]{%
  \href{http://www.ams.org/mathscinet-getitem?mr=#1}{#2}
}
\providecommand{\href}[2]{#2}
\begin{thebibliography}{MFK94}

\bibitem[AOV08]{tame}
Dan Abramovich, Martin Olsson, and Angelo Vistoli, \emph{Tame stacks in
  positive characteristic}, Ann. Inst. Fourier (Grenoble) \textbf{58} (2008),
  no.~4, 1057--1091.

\bibitem[Art74]{artin_versal}
M.~Artin, \emph{Versal deformations and algebraic stacks}, Invent. Math.
  \textbf{27} (1974), 165--189.

\bibitem[BB63]{bb_homogeneous}
A.~Bia{\l}ynicki-Birula, \emph{On homogeneous affine spaces of linear algebraic
  groups}, Amer. J. Math. \textbf{85} (1963), 577--582.

\bibitem[Cap94]{caporaso}
Lucia Caporaso, \emph{A compactification of the universal {P}icard variety over
  the moduli space of stable curves}, J. Amer. Math. Soc. \textbf{7} (1994),
  no.~3, 589--660.

\bibitem[Con05]{conrad}
Brian Conrad, \emph{Keel-mori theorem via stacks},
  \url{http://www.math.lsa.umich.edu/~bdconrad/papers/coarsespace.pdf} (2005).

\bibitem[DM69]{deligne-mumford}
P.~Deligne and D.~Mumford, \emph{The irreducibility of the space of curves of
  given genus}, Inst. Hautes \'Etudes Sci. Publ. Math. (1969), no.~36, 75--109.



\bibitem[EGA]{ega}
A.~Grothendieck, \emph{\'{E}l\'ements de g\'eom\'etrie alg\'ebrique}, Inst.
  Hautes \'Etudes Sci. Publ. Math. (1961-1967), no.~4,8,111,17,20,24,28,32.


\bibitem[FC90]{faltings-chai}
Gerd Faltings and Ching-Li Chai, \emph{Degeneration of abelian varieties},
  Ergebnisse der Mathematik und ihrer Grenzgebiete (3) [Results in Mathematics
  and Related Areas (3)], vol.~22, Springer-Verlag, Berlin, 1990, With an
  appendix by David Mumford.

\bibitem[Fog83]{fogarty}
John Fogarty, \emph{Geometric quotients are algebraic schemes}, Adv. in Math.
  \textbf{48} (1983), no.~2, 166--171.

\bibitem[Fog87]{fogarty2}
\bysame, \emph{Finite generation of certain subrings}, Proc. Amer. Math. Soc.
  \textbf{99} (1987), no.~1, 201--204.

\bibitem[Gie77]{gieseker_surface_bundles}
D.~Gieseker, \emph{On the moduli of vector bundles on an algebraic surface},
  Ann. of Math. (2) \textbf{106} (1977), no.~1, 45--60.
  
  \bibitem[GIT]{git}
David Mumford, \emph{Geometric invariant theory}, Ergebnisse der Mathematik und
  ihrer Grenzgebiete, Neue Folge, Band 34, Springer-Verlag, Berlin, 1965.

\bibitem[Hab78]{haboush_stab}
W.~J. Haboush, \emph{Homogeneous vector bundles and reductive subgroups of
  reductive algebraic groups}, Amer. J. Math. \textbf{100} (1978), no.~6,
  1123--1137.

\bibitem[Has05]{hassett_genus2}
Brendan Hassett, \emph{Classical and minimal models of the moduli space of
  curves of genus two}, Geometric methods in algebra and number theory, Progr.
  Math., vol. 235, Birkh\"auser Boston, Boston, MA, 2005, pp.~169--192.

\bibitem[HH06]{hassett-hyeon_contraction}
Brendan Hassett and Donghoon Hyeon, \emph{Log canonical models for the moduli
  space of curves: First divisorial contraction}, math.AG/0607477 (2006).

\bibitem[HH08]{hassett-hyeon_flip}
\bysame, \emph{Log minimal model program for the moduli space of stable curves:
  The first flip}, math.AG/0806.3444 (2008).

\bibitem[HL97]{huybrechts-lehn}
Daniel Huybrechts and Manfred Lehn, \emph{The geometry of moduli spaces of
  sheaves}, Aspects of Mathematics, E31, Friedr. Vieweg \& Sohn, Braunschweig,
  1997.

\bibitem[HL07a]{hyeon-lee_genus3}
Donghoon Hyeon and Yongnam Lee, \emph{Log minimal model program for the moduli
  space of stable curves of genus three}, math.AG/0703093 (2007).

\bibitem[HL07b]{hyeon-lee_genus2}
Donghoon Hyeon and Yongnam Lee, \emph{Stability of tri-canonical curves of
  genus two}, Math. Ann. \textbf{337} (2007), no.~2, 479--488.

\bibitem[KKV89]{kkv}
Friedrich Knop, Hanspeter Kraft, and Thierry Vust, \emph{The {P}icard group of
  a {$G$}-variety}, Algebraische Transformationsgruppen und Invariantentheorie,
  DMV Sem., vol.~13, Birkh\"auser, Basel, 1989, pp.~77--87.

\bibitem[KM97]{keel-mori}
Se{\'a}n Keel and Shigefumi Mori, \emph{Quotients by groupoids}, Ann. of Math.
  (2) \textbf{145} (1997), no.~1, 193--213.

\bibitem[Knu71]{Knutson}
Donald Knutson, \emph{Algebraic spaces}, Springer-Verlag, Berlin, 1971, Lecture
  Notes in Mathematics, Vol. 203.

\bibitem[Kol97]{kollar_quotients}
J{\'a}nos Koll{\'a}r, \emph{Quotient spaces modulo algebraic groups}, Ann. of
  Math. (2) \textbf{145} (1997), no.~1, 33--79.

\bibitem[Kra89]{kraft}
Hanspeter Kraft, \emph{{$G$}-vector bundles and the linearization problem},
  Group actions and invariant theory (Montreal, PQ, 1988), CMS Conf. Proc.,
  vol.~10, Amer. Math. Soc., Providence, RI, 1989, pp.~111--123.

\bibitem[Lie07]{lieblich_twisted}
Max Lieblich, \emph{Moduli of twisted sheaves}, Duke Math. J. \textbf{138}
  (2007), no.~1, 23--118.

\bibitem[LMB00]{lmb}
G{\'e}rard Laumon and Laurent Moret-Bailly, \emph{Champs alg\'ebriques},
  Ergebnisse der Mathematik und ihrer Grenzgebiete. 3. Folge. A Series of
  Modern Surveys in Mathematics [Results in Mathematics and Related Areas. 3rd
  Series. A Series of Modern Surveys in Mathematics], vol.~39, Springer-Verlag,
  Berlin, 2000.

\bibitem[Lun73]{luna}
Domingo Luna, \emph{Slices \'etales}, Sur les groupes alg\'ebriques, Soc. Math.
  France, Paris, 1973, pp.~81--105. Bull. Soc. Math. France, Paris, M\'emoire
  33.

\bibitem[Mar77]{maruyama_sheaves}
Masaki Maruyama, \emph{Moduli of stable sheaves. {I}}, J. Math. Kyoto Univ.
  \textbf{17} (1977), no.~1, 91--126. \MR{MR0450271 (56 \#8567)}

\bibitem[Mat60]{matsushima}
Yoz{\^o} Matsushima, \emph{Espaces homog\`enes de {S}tein des groupes de {L}ie
  complexes}, Nagoya Math. J \textbf{16} (1960), 205--218.

\bibitem[Mel07]{melo}
Margarida Melo, \emph{Compactified picard stacks over {$\bar{\mathcal M}_g$}},
  math.AG/0710.3008 (2007).

\bibitem[MFK94]{git3}
D.~Mumford, J.~Fogarty, and F.~Kirwan, \emph{Geometric invariant theory}, third
  ed., Ergebnisse der Mathematik und ihrer Grenzgebiete (2) [Results in
  Mathematics and Related Areas (2)], vol.~34, Springer-Verlag, Berlin, 1994.


\bibitem[Nag59]{nagata_hilbert14}
Masayoshi Nagata, \emph{On the {$14$}-th problem of {H}ilbert}, Amer. J. Math.
  \textbf{81} (1959), 766--772.

\bibitem[Nag62]{nagata_complete}
\bysame, \emph{Complete reducibility of rational representations of a matric
  group.}, J. Math. Kyoto Univ. \textbf{1} (1961/1962), 87--99.

\bibitem[Nag64]{nagata_invariants-affine}
\bysame, \emph{Invariants of a group in an affine ring}, J. Math. Kyoto Univ.
  \textbf{3} (1963/1964), 369--377.

\bibitem[Ols07]{Olsson-sheaves}
Martin Olsson, \emph{Sheaves on {A}rtin stacks}, J. Reine Angew. Math.
  \textbf{603} (2007), 55--112.

\bibitem[RG71]{raynaud-gruson}
Michel Raynaud and Laurent Gruson, \emph{Crit\`eres de platitude et de
  projectivit\'e. {T}echniques de ``platification'' d'un module}, Invent. Math.
  \textbf{13} (1971), 1--89.

\bibitem[Ric77]{richardson}
R.~W. Richardson, \emph{Affine coset spaces of reductive algebraic groups},
  Bull. London Math. Soc. \textbf{9} (1977), no.~1, 38--41.

\bibitem[Ryd07]{rydh_quotients}
David Rydh, \emph{Existence of quotients by finite groups and coarse moduli
  spaces}, math.AG/0708.3333 (2007).

\bibitem[Ryd08]{rydh_approx}
\bysame, \emph{Noetherian approximation of algebraic spaces and stacks}, in
  preparation (2008).

\bibitem[Sch91]{schubert}
David Schubert, \emph{A new compactification of the moduli space of curves},
  Compositio Math. \textbf{78} (1991), no.~3, 297--313.

\bibitem[Ses77]{seshadri_reductivity}
C.~S. Seshadri, \emph{Geometric reductivity over arbitrary base}, Advances in
  Math. \textbf{26} (1977), no.~3, 225--274.

\bibitem[Ses82]{seshadri_bundles}
\bysame, \emph{Fibr\'es vectoriels sur les courbes alg\'ebriques},
  Ast\'erisque, vol.~96, Soci\'et\'e Math\'ematique de France, Paris, 1982,
  Notes written by J.-M. Drezet from a course at the \'Ecole Normale
  Sup\'erieure, June 1980.

\bibitem[SGA3]{sga3}
\emph{Sch\'emas en groupes}, S\'eminaire de G\'eom\'etrie Alg\'ebrique du Bois
  Marie 1962/64 (SGA 3). Dirig\'e par M. Demazure et A. Grothendieck. Lecture
  Notes in Mathematics, Vol. 151,152,153, Springer-Verlag, Berlin, 1962/1964.

\bibitem[Sim94]{simpson_sheaves}
Carlos~T. Simpson, \emph{Moduli of representations of the fundamental group of
  a smooth projective variety. {I}}, Inst. Hautes \'Etudes Sci. Publ. Math.
  (1994), no.~79, 47--129.

\bibitem[Vis05]{vistoli_fga}
Angelo Vistoli, \emph{Grothendieck topologies, fibered categories and descent
  theory}, Fundamental algebraic geometry, Math. Surveys Monogr., vol. 123,
  Amer. Math. Soc., Providence, RI, 2005, pp.~1--104.

\end{thebibliography}

\bibliographystyle{amsalpha}

\end{document}